\documentclass[10pt]{article}
\usepackage{newcent}
\usepackage{url}
\usepackage{amsmath,amssymb,amsthm,graphicx}
\usepackage{enumerate}
\def\a{\alpha}
\def\b{\beta}

\def\la{\lambda}
\def\k{\kappa}
\def\ph{\varphi}

\def\O{\mathcal{O}}
\def\W{\mathcal{W}}

\def\N{\mathbb{N}}

\def\R{\mathbb{R}}

\def\Z{\mathbb{Z}}

\def\PII{\mbox{\rm P$_{\rm II}$}}

\def\Ai{\mathop{\rm Ai}\nolimits}
\def\Bi{\mathop{\rm Bi}\nolimits}
\def\d{{\rm d}}
\def\e{{\rm e}}
\def\i{\ifmmode{\rm i}\else\char"10\fi}
\def\p{Painlev\'e}
\def\w{\omega}

\newcommand{\HyperpFq}[2]{{}_{#1}F_{#2}}
\def\v{v}
\def\ww#1{\w(#1;t)}
\def\vx{v'(x)}
\def\vy{v'(y)}

\def\ep{\varepsilon}
\def\ds{\displaystyle}
\def\dx{\d x}\def\dy{\d y}
\def\Kxy{\mathcal{K}(x,y)}
\def\intS{\int_0^{\infty}}
\def\imp{\int_{-\infty}^{\infty}}
\def\intS{\int_{0}^{\infty}}
\def\etal{\textit{et al.}}

\newcommand{\deriv}[3][]{\frac{\d^{#1}{#2}}{{\d{#3}}^{#1}}}

\newtheorem{theorem}{Theorem}[section]

\newtheorem{lemma}[theorem]{Lemma}
\newtheorem{corollary}[theorem]{Corollary}
\theoremstyle{definition}

\newtheorem{remark}[theorem]{Remark}

\newtheorem{conjecture}[theorem]{Conjecture}
\numberwithin{figure}{section}
\numberwithin{equation}{section}
\numberwithin{table}{section}

\def\ode{ordinary differential equation}

\def\beq{\begin{equation}}
\def\eeq{\end{equation}}
\newcommand{\comment}[1]{}
\def\fig#1{\includegraphics[width=2in]{#1}}

\begin{document}
\title{Generalised Airy Polynomials}

\author{Peter A. Clarkson$^{1}$ and Kerstin Jordaan$^{2}$\\[2.5pt]
$^{1}$ School of Mathematics, Statistics and Actuarial Science\\ University of Kent, Canterbury, CT2 7FS, UK\\ \texttt{P.A.Clarkson@kent.ac.uk}\\ Orcid: 0000-0002-8777-5284\\[2.5pt]
$^{2}$ Department of Decision Sciences\\
University of South Africa, Pretoria, 0003, South Africa\\ 
\texttt{jordakh@unisa.ac.za}\\ Orcid: 0000-0002-1675-5366}

\maketitle
\begin{itemize}
\item[]\textbf{Keywords}: Semi-classical orthogonal polynomials, generalised Airy weight; generalised sextic Freud weight; moments; recurrence coefficients; zeros; asymptotics
\item[]\textbf{Mathematics Subject Classification (2010)}: 33C47, 34E99, 42C05, 65Q99
\end{itemize} 

\begin{abstract}
We consider properties of semi-classical orthogonal polynomials with respect to the generalised Airy weight 
\[\omega(x;t,\lambda)=x^{\lambda}\exp\left(-\tfrac13x^3+tx\right),\qquad x\in \mathbb{R}^+\] with parameters $\lambda>-1$ and $t\in \mathbb{R}$. We also investigate the zeros and recurrence coefficients of the polynomials. The generalised sextic Freud weight
\[\omega(x;t,\lambda)=|x|^{2\lambda+1}\exp\left(-x^6+tx^2\right), \qquad x\in \mathbb{R}\] arises from a symmetrisation of the generalised Airy weight and we study analogous properties of the polynomials orthogonal with respect to this weight. 
\end{abstract}
\section{Introduction}

{Suppose $P_n(x)$, for $n\in\N$, 
is a sequence of \textit{classical} orthogonal polynomials (such as Hermite, Laguerre and Jacobi polynomials), then $P_n(x)$ is a solution of a second-order \ode\ of the form}
\beq \label{eq:Pn}
\sigma(x)\deriv[2]{P_n}{x}+\tau(x)\deriv{P_n}{x}=\lambda_nP_n
\eeq 
where $\sigma(x)$ is a monic polynomial with deg$(\sigma)\leq2$, $\tau(x)$ is a polynomial with deg$(\tau)=1$, and $\lambda_n$ is a real number which depends on the degree of the polynomial solution, see Bochner \cite{refBochner}. Equivalently, the weights of classical orthogonal polynomials satisfy a first-order \ode, the \textit{Pearson equation}
\beq \label{eq:Pearson}
\deriv{}{x}[\sigma(x)\w(x)]=\tau(x)\w(x)
\eeq 
with $\sigma(x)$ and $\tau(x)$ the same polynomials as in \eqref{eq:Pn}, see, for example \cite{refAlNod,refBochner,refChihara}. 
\comment{The weights of classical orthogonal polynomials satisfy a first-order ordinary differential equation, the \textit{Pearson equation}
\beq\label{eq:Pearson}
 \deriv{}{x}[\sigma(x)\w(x)]=\tau(x)\w(x)
 \eeq
 where $\sigma(x)$ is a monic polynomials of degree at most $2$ and $\tau(x)$ is a polynomial with degree $1$.}%

For \textit{semi-classical} orthogonal polynomials, the weight function $\w(x)$ satisfies the Pearson equation \eqref{eq:Pearson} with either deg$(\sigma)>2$ or deg$(\tau)\neq 1$ (cf.~\cite{refHvR,refMaroni}). 
For example, the generalised Airy weight 
\beq\label{genAiry}
 \w(x;t,\la)=x^{\la}\exp\left(-\tfrac13x^3+tx\right), 
\eeq with parameters $\lambda>-1$ and $t\in \mathbb{R}$,
satisfies the Pearson equation \eqref{eq:Pearson} with 
 \[\sigma(x)=x,\qquad\tau(x)=-x^3+tx+\la+1\] 
and the generalised sextic Freud weight
\beq\label{freud6g}
\w(x;t,\la)=|x|^{2\la+1}\exp\left(-x^6+tx^2\right),\qquad x\in \R\eeq 
with $\la>-1$ and $t\in \R$ parameters,
satisfies \eqref{eq:Pearson} with \[\sigma(x)=x,\qquad\tau(x)=2\la+2+2tx^2-6x^6.\] 
Orthogonal polynomials associated with the exponential cubic weight 
\beq \w(x)=\exp(-x^3),\qquad x\in\mathcal{C}\label{cubic}\eeq
where $\mathcal{C}$ a contour in the complex plane, were investigated in \cite{refDeano,refMFS,refWVAFZ}, 
while the semiclassical weight 
\beq\w(x;t)=\exp\left(-\tfrac13x^3+tx\right),\qquad x\in\mathcal{C}\label{tcubic}\eeq with $t\in\R$ and $\mathcal{C}$ is a contour in the complex plane, was discussed in \cite{refBD13,refBD16,refBDY,refCAiry,refCLVA,refDeano,refMagnus95}. These studies of the weights \eqref{cubic} and \eqref{tcubic} are for contours in the complex plane. In contrast, in this paper, we study orthogonal polynomials associated with an exponential cubic weight on the positive real axis.

 We are concerned with semi-classsical orthogonal polynomials associated with the generalised Airy weight \eqref{genAiry} as well as generalised sextic Freud polynomials associated with the weight function \eqref{freud6g} that arise from a symmetrisation of generalised Airy polynomials. In \S\ref{sec:genairy} we consider some properties of generalised Airy polynomials, their recurrence coefficients and their zeros. In \S\ref{sec:rc} we consider the recurrence coefficients of polynomials orthogonal with respect to \eqref{genAiry} and correct some of the results in Wang \etal\ \cite{refWZC20b}. Properties of generalised sextic Freud polynomials and their zeros are considered in \S\ref{sec:gen6freud}. These polynomials were also investigated in \cite{refCJ20,refWZC20a} and we correct one of the results in \cite{refWZC20a} in \S \ref{gen6feq}.
  
\section{Orthogonal polynomials}
Let $\mu$ be a positive Borel measure with support $S$ defined on $\R$ for which moments of all orders exist, that is
\begin{equation}\label{eq:moment}\mu_n=\int_{S}  x^n\, \d\mu(x), \quad n=0,1,2\dots.\end{equation}
When the sequence $\{\mu_n\}_{n\geq 0}$ is positive, that is for all $n\in\{0,1,2,..\}$ the Hankel determinant
 \begin{equation}\label{eq:dets}
 \Delta_n=\left|\begin{array}{cccc}
 \mu_0&\mu_1& \ldots &\mu_{n-1}\\
 \mu_1&\mu_2&\ldots&\mu_{n}\\
\vdots & \vdots& \ddots &\vdots\\
 \mu_{n-1}&\mu_{n}&\ldots &\mu_{2n-2}
 \end{array}
 \right|,\qquad n\geq1
 \end{equation}
 is positive, the family of monic polynomials 
 
 \begin{equation*}
 P_n(x)=\frac{1}{\Delta_n}\left|\begin{array}{cccc}
 \mu_0&\mu_1&\ldots&\mu_n\\
 \mu_1&\mu_2&\ldots&\mu_{n+1}\\
\vdots & \vdots& \ddots &\vdots\\
 \mu_{n-1}&\mu_{n}&\ldots &\mu_{2n-1}\\
 1&x&\ldots &x^n
 \end{array}
 \right|
 \end{equation*}
 is orthogonal with respect to the measure $\mu$ on its support, i.e. \beq \nonumber\int_S P_m(x)P_n(x)\,\d \mu(x) = h_n\delta_{m,n},\qquad h_n>0\label{eq:norm}\eeq
 where $\delta_{m,n}$ denotes the Kronecker delta. The sequence $\{P_n(x)\}_{n=0}^{\infty}$ satisfies the three-term recurrence relation 
  \begin{equation}\label{eq:3trr}xP_n(x)=P_{n+1}(x)+\a_nP_n(x)+\b_{n}P_{n-1}(x)\end{equation}
 with initial conditions $P_{-1}(x)=0$ and $P_0(x)=1$ and where
\beq\label{eq:anbn}  \a_n = \frac{1}{h_n}\int_S xP_n^2(x) \d\mu(x),\qquad
 \b_n =\frac{\Delta_{n-1}\Delta_{n+1}}{\Delta_n^2} >0.\eeq Additional information about orthogonal polynomials can be found in, for example, \cite{refChihara,refIsmail,refSzego}.

Now suppose that the measure is absolutely continuous and has the form
\[\d \mu(x)=\w(x;t,\la)\,\dx\]
where \beq \label{scweight}\w(x;t,\la)=x^{\la}\w_0(x)\exp(xt),\qquad x\in\R^+,\qquad\la>-1\eeq 
with finite moments for all $t\in\R$, which is the case for the generalised Airy weight \eqref{genAiry}.
If the weight has the form \eqref{scweight}, which depends on the parameters $t$ and $\la$, then the orthogonal polynomials $P_n$, 
the recurrence coefficients $\a_n$, $\b_n$ given by \eqref{eq:anbn},
the determinant $\Delta_n$ given by \eqref{eq:dets} and 
the moments $\mu_k$ given by \eqref{eq:moment} are now functions of $t$ and $\la$. 
Specifically, in this case then
\begin{equation}\label{mu0}
\mu_k(t;\la) 
= \intS x^{k+\la} \w_0(x)\exp(xt)\,\d x 
=\deriv[k]{}{t}\left( \intS x^{\la}\w_0(x)\exp(xt)\,\d x\right)
=\deriv[k]{\mu_0(t;\la)}{t}
.\end{equation} Further, the recurrence relation has the form
\beq \label{eq:scrr}xP_n(x;t,\la)=P_{n+1}(x;t,\la)+\a_n(t;\la)P_n(x;t,\la)+\b_n(t;\la)P_{n-1}(x;t,\la)\eeq 
where we have explicitly indicated that the coefficients $\a_n(t;\la)$ and $\b_n(t;\la)$ depend on $t$ and $\la$.

\begin{theorem}{\label{thm:2.1}If the weight has the form \eqref{scweight}, then the determinant $\Delta_n(t;\la)$  given by \eqref{eq:dets} can be written as 
\begin{align}\label{scHank}
\Delta_n(t;\la)&=\W\left(\mu_0,\deriv{\mu_0}t,\ldots,\deriv[n-1]{\mu_0}t\right).
\end{align}
where $\W(\ph_1,\ph_2,\ldots,\ph_n)$ is the Wronskian given by
\[\W(\ph_1,\ph_2,\ldots,\ph_n)=\left|\begin{matrix} 
\ph_1 & \ph_2 & \ldots & \ph_n\\
\ph_1^{(1)} & \ph_2^{(1)} & \ldots & \ph_n^{(1)}\\
\vdots & \vdots & \ddots & \vdots \\
\ph_1^{(n-1)} & \ph_2^{(n-1)} & \ldots & \ph_n^{(n-1)}
\end{matrix}\right|,\qquad \ph_j^{(k)}=\deriv[k]{\ph_j}{t}.\]
}\end{theorem}

\begin{proof}{See, for example, \cite[Theorem 2.1]{refCJ14} }\end{proof}

The Hankel determinant $\Delta_n(t;\la)$ satisfies the Toda equation, as shown in the following theorem.

\begin{theorem}{\label{thm:toda}The Hankel determinant $\Delta_n(t;\la)$ given by \eqref{scHank} satisfies the Toda equation
\[
\deriv[2]{}{t}\ln\Delta_n(t;\la)=\frac{\Delta_{n-1}(t;\la)\Delta_{n+1}(t;\la)}{\Delta_{n}^2(t;\la)}.
\]
}\end{theorem}
\begin{proof}{See, for example, Nakamira and Zhedanov \cite[Proposition 1]{refNZ}; also \cite{refCI97}.}\end{proof}

Using Theorems \ref{thm:2.1} and \ref{thm:toda}, we can express the recurrence coefficients $\a_n(t;\la)$ and $\b_n(t;\la)$ in terms of derivatives of the Hankel determinant $\Delta_n(t;\la)$ and so obtain explicit expressions for these coefficients.

\begin{theorem}{\label{thm:anbn}The coefficients $\a_n(t;\la)$ and $\b_n(t;\la)$ in the recurrence relation \eqref{eq:scrr}
associated with monic polynomials orthogonal with respect to a weight of the form \eqref{scweight} are given by
\[ \a_n(t;\la) =\deriv{}{t}\ln \frac{\Delta_{n+1}(t;\la)}{\Delta_{n}(t;\la)},\qquad 
\b_n(t;\la) =
\deriv[2]{}{t}\ln\Delta_n(t;\la) \]
with $\Delta_n(t;\la)$ is the Hankel determinant given by \eqref{scHank}.
}\end{theorem}

\begin{proof}{See Chen and Ismail \cite{refCI97}.}\end{proof}

\comment{Equivalently the recurrence coefficients $\a_n(t;\la)$ and $\b_n(t;\la)$ can be expressed in terms of $h_n(t;\la)$ given by \eqref{def:norm}.
\begin{lemma}{\label{thm:anbn2}The coefficients $\a_n(t;\la)$ and $\b_n(t;\la)$ in the recurrence relation \eqref{eq:scrr}
associated with monic polynomials orthogonal with respect to a weight of the form \eqref{scweight} are given by
\beq \label{def:anbn2} \a_n(t;\la) =\deriv{}{t}\ln h_{n}(t;\la),\qquad 
\b_n(t;\la) =\frac{h_{n+1}(t;\la)}{h_{n}(t;\la)}\eeq 
where $h_n(t;\la)$ is given by \eqref{def:norm}.
}\end{lemma}
\begin{proof}{See Chen and Ismail \cite{refCI97}.}\end{proof}}%

Additionally the coefficients $\a_n(t;\la)$ and $\b_n(t;\la)$ in the recurrence relation \eqref{eq:scrr} satisfy a Toda system.

\begin{theorem}{\label{thm:todasys}
The coefficients $\a_n(t;\la)$ and $\b_n(t;\la)$ in the recurrence relation \eqref{eq:scrr} associated with a weight of the form \eqref{scweight} satisfy the Toda system
\[ \deriv{\a_n}t=\b_{n+1}-\b_n,\qquad\deriv{\b_n}t=\b_n(\a_n-\a_{n-1}).\]
}\end{theorem}
\begin{proof}{See Chen and Ismail \cite{refCI97}, Ismail \cite[\S2.8, p.\ 41]{refIsmail}; see also \cite{refFVAZ} for further details.}\end{proof}

\section{Generalised Airy polynomials} \label{sec:genairy}
In this section we are concerned with the generalised Airy weight \eqref{genAiry}
and the polynomials orthogonal with respect to this weight.

\begin{lemma}\label{lem:Freud6weight}
For the generalised Airy weight \eqref{genAiry},
the first moment is given by
\begin{align}
\mu_0(t;\la)& =\intS x^\la\exp\left(-\tfrac13x^3+tx\right)\,\d x \nonumber\\
& = 3^{(\la-2)/3}\,\Gamma(\tfrac13\la+\tfrac13) \;\HyperpFq12(\tfrac13\la+\tfrac13;\tfrac13,\tfrac23;\tfrac19t^3) 
+ 3^{(\la-1)/3}\,t\,\Gamma(\tfrac13\la+\tfrac23) \;\HyperpFq12(\tfrac13\la+\tfrac23;\tfrac23,\tfrac43;\tfrac19t^3)\nonumber\\ &\qquad\qquad
+ \tfrac12\,3^{\la/3}\,t^2\,\Gamma(\tfrac13\la+1) \;\HyperpFq12(\tfrac13\la+1;\tfrac43,\tfrac53;\tfrac19t^3) 
\label{eq:mu0}\end{align}
where 
$\HyperpFq12(a_1;b_1,b_2;z)$ is the generalised hypergeometric function. Further, $\ph(t)=\mu_0(t;\la)$ satisfies the third order equation
\[\deriv[3]{\ph}{t}-t\deriv{\ph}{t}-(\la+1)\ph=0.\]
\end{lemma}
\begin{proof}See \cite[Lemma 3.1]{refCJ20}.
\end{proof}

\begin{remark}{\label{rmks32}\rm 
\begin{enumerate}\item[]
\item If $\la=-\tfrac12$ then the first moment is given by
\[ \mu_0(t;-\tfrac12)=\intS x^{-1/2}\exp\left(-\tfrac13x^3+tx\right)\,\d x = \pi^{3/2}2^{-1/3}[\Ai^2(\tau)+\Bi^2(\tau)],\qquad \tau = 2^{-2/3}t\]
where $\Ai(\tau)$ and $\Bi(\tau)$ are the Airy functions. This result is equation 9.11.4 in the DLMF \cite{refNIST}, which is due to Muldoon \cite[p32]{refMul77}. 
\item For the generalised Airy weight \eqref{genAiry} the $k$-th moment is given by \[\mu_k(t;\la) =\intS x^{\la+k}\exp\left(-\tfrac13x^3+tx\right)\,\d x=\mu_0(t,\la+k)\] which, using \eqref{mu0}, implies that \[\deriv[k]{\mu_0}{t}=\mu_0(t,\la+k)\]
 
\end{enumerate}
}\end{remark} 
\comment{\item The generalised Airy weight \eqref{genAiry} is an example of a semi-classical weight for which the first moment $\mu_0(t;\la)$ satisfies a \textit{third order equation}. In our earlier studies of semi-classical weights \cite{refCJ14,refCJ18,refCJK}, the first moment has satisfied a second order equation. For example, for the semi-classical Laguerre weight
\beq \w(x;t)=x^\la\exp(-x^2+tx),\qquad x\in\R^+\label{eq:scLag1}\eeq the first moment is expressed in terms of parabolic cylinder functions $D_{\nu}(z)$, see \cite{refCJ14}.
This is a classical special function that satisfies a second order equation.
\item Equation \eqref{eq;phi} arises in association with threefold symmetric Hahn-classical multiple orthogonal polynomials \cite{refLVA} and in connection with Yablonskii--Vorob'ev polynomials associated with rational solutions of the second \p\ equation \cite{refCM03}.
\end{enumerate}\end{remarks}}

From Theorem \ref{thm:anbn}, we have the following representations of $\a_n(t;\la)$ and $\b_n(t;\la)$.
\begin{theorem}{The coefficients $\a_n(t;\la)$ and $\b_n(t;\la)$ in the recurrence relation \eqref{eq:scrr}
associated with monic polynomials orthogonal with respect to the generalised Airy weight \eqref{genAiry} are given by
\[ \a_n(t;\la) =\deriv{}{t}\ln \frac{\Delta_{n+1}(t;\la)}{\Delta_{n}(t;\la)},\qquad 
\b_n(t;\la) =
\deriv[2]{}{t}\ln\Delta_n(t;\la) \]
where $\Delta_n(t;\la)$ is the Hankel determinant given by 
\[ \Delta_n(t;\la)=\W\left(\mu_0,\deriv{\mu_0}t,\ldots,\deriv[n-1]{\mu_0}t\right)\]
with $\mu_0(t;\la)$ given by \eqref{eq:mu0}.
}\end{theorem}
\subsection{Differential and discrete equations satisfied by generalised Airy polynomials}
\label{sec:genairyprop}
We derive a differential-difference equation,  a differential equation and a mixed recurrence relation satisfied by generalised Airy 
polynomials. 

The coefficients ${A}_n(x)$ and ${B}_n(x)$ in the relation
\beq \label{ddee}\deriv{P_n}{x}(x;t,\la)=\beta_n(t;\la){A}_n(x)P_{n-1}(x;t,\la)-{B}_n(x)P_n(x;t,\la)\eeq 
satisfied by semi-classical orthogonal polynomials can be derived using a technique introduced by Shohat \cite{refShohat39} for weights $\w(x)$ such that $\displaystyle{\w'(x)}/{\w(x)}$ is a rational function. The method of ladder operators was introduced by Chen and Ismail in \cite{refCI97}, see also \cite[Theorem 3.2.1]{refIsmail} and adapted in \cite{refChenFeigin} for the situation where the weight function vanishes at one point. Explicit expressions for the coefficients in the differential-difference equation \eqref{ddee} when the weight function is positive on the real line except for one point are provided in \cite{refCJK}. The coefficients in the differential-difference relation for generalised Airy polynomials associated with the weight \eqref{genAiry} are given in the next result.

\begin{theorem}\label{thm:dde}{For the generalised Airy weight \eqref{genAiry}
the monic orthogonal polynomials $P_{n}(x;t,\la)$ with respect to this weight satisfy the differential-difference equation \eqref{ddee} with
\begin{subequations}\label{AnBna}\begin{align}
A_n(x)&= x + \a_n + \frac{R_n}{x}\\
B_n(x)&=\b_n + \frac{r_n}{x}
\end{align}\end{subequations}
where $\a_n(t;\la)$ and $\b_n(t;\la)$ are the coefficients in the three-term recurrence relation \eqref{eq:scrr} and 
\begin{subequations}\label{def:Rnrn}\begin{align}R_n&=\a_n^2+\b_n+\b_{n+1}-t\\r_n&=\tfrac12 \left(\la -\a_{n+1} \b_{n+1}+\a_{n-1} \b_n-\a_n^3+t \a_n+1\right)-\a_n \b_{n+1}\label{def:Rnrnb}.\end{align}\end{subequations}}\end{theorem}
 \begin{proof}Since $P_n(x)=P_n(x;t,\la)$ is a polynomial of degree $n$, we can write
 \begin{align}\label{quasi}
 x\deriv{P_n}{x}(x) &= \sum_{k=0}^{n}c_{n,k} P_k(x).
 \end{align}
 Multiplying \eqref{quasi} by $P_k(x)\,\w(x)$, integrating both sides of the equation with respect to $x$ and applying the orthogonality relation, yields, for $k=0,1,2,...,n,$
 \begin{align}\label{Ccoeff}
 h_kc_{n,k} &= \intS x\deriv{P_n}{x}(x) P_k(x)\,\w(x) \,\d{x}, \qquad h_k\neq0.
 \end{align}
 Integrating the right hand side of \eqref{Ccoeff} by parts, we obtain, for $k=0,1,2,\dots n$, 
 \begin{align}\nonumber
 h_k c_{n,k} &=\Big[ xP_k(x) P_n(x)\,\w(x;t,\la)\Big]_{0}^{\infty} 
 - \int_{0}^\infty \deriv{}{x}\left[ x P_k (x)\,\w(x;t,\la) \right]P_n(x) \,\d{x} \nonumber\\
 &=- \intS \left[P_n(x) P_k(x) 
 + xP_n(x) \deriv{P_k}{x}(x)\right]\w(x;t,\la) \,\d{x}
 - \intS xP_n(x) P_k(x) \deriv{\w}{x}(x) \,\d{x}
 \nonumber\\
 &=- \intS \left[P_n(x) P_k(x) 
 + xP_n(x) \deriv{P_k}{x}(x)\right]\w(x;t,\la) \,\d{x}
 - \intS P_n(x) P_k(x) (\la+tx-x^3)\,\w(x;t,\la) \,\d{x}.
\label{maineq2}
 \end{align}
 For $k=n$, it follows from \eqref{maineq2}, that
 \begin{align}
 h_n c_{n,n} &= \intS x\deriv{P_n}{x}(x) P_n(x)\,\w(x;t,\la) \,\d{x}\nonumber\\&=-\tfrac12\intS P_n^2(x) \w(x;t,\la) \,\dx -\tfrac12\intS  P_n^2(x) (\la+tx-x^3)\,\w(x;t,\la)\,\d{x} \nonumber
 \\&= - \tfrac12h_n -\tfrac12\la h_n-\tfrac 12t\intS xP_n^2(x)\,\w(x;t,\la) \,\d{x}+\tfrac 12\intS x^3P_n^2(x)\,\w(x;t,\la)\,\d{x}.\label{eq:cnn1}
 \end{align}
 Iterating the three-term recurrence relation \eqref{eq:scrr}, yields
 \begin{align}
 x^3P_n(x) = P_{n+3}(x) &+ (\a_{n+2}+ \a_{n}+ \a_{n+1})P_{n+2}(x) \nonumber\\& 
 +(\a_n^2+\a_{n+1} \a_n+\a_{n+1}^2+\b_n+\b_{n+1}+\b_{n+2})P_{n+1}(x)
 \nonumber\\& 
 + (2 \a_n \left(\b_n+\b_{n+1}\right)+\a_{n-1} \b_n+\a_{n+1} \b_{n+1}+\a_n^3)P_n(x)
 \nonumber\\& +\b_n \left(\a_{n-1}^2+\a_n \a_{n-1}+\a_n^2+\b_{n-1}+\b_n+\b_{n+1}\right)P_{n-1}(x)
 \nonumber \\&
 +\b_{n-1}\b_n(\a_n+\a_{n-1}+\a_{n-2})P_{n-2}(x)+ \b_n\b_{n-1}\b_{n-2}P_{n-3}(x).\label{recurrence3}
 \end{align}
 Substituting \eqref{eq:scrr} and \eqref{recurrence3} into \eqref{eq:cnn1} it follows that 
 \begin{align} c_{n,n}&= -\tfrac12(\a_nt+\la+1-\a_{n+1}\b_{n+1}-\a_n^3-\a_{n-1}\b_n)+\a_n(\b_{n+1}+\b_{n})
 \label{eq:cnn4}\end{align}
For $k=0,1,2,\dots n-1$, \eqref{maineq2} yields
 \begin{align}\nonumber
 h_k c_{n,k} 
 &= -\intS {P_n(x) P_k(x) \left(\la+tx -x^{3}\right) }\,\w(x;t,\la)\,\d{x}\nonumber
 \\&= 
 \intS \big(x^3-tx\big) P_n(x) P_k(x)\,\w(x;t,\la) \,\d{x}.\label{maineq}
 \end{align}
 Substituting \eqref{recurrence3} and \eqref{eq:3trr} into \eqref{maineq} we see that $c_{n,n-j} =0$ for $j=0,1\dots,n-4$ while
 \begin{subequations}\label{Aau}
 \begin{align} 
 c_{n,n-1} & = \b_{n}(\b_{n+1}+\b_n+\b_{n-1}+\a_n^2+\a_{n-1}\a_n+\a_{n-1}^2-t)\\
 c_{n,n-2} & =\b_n\b_{n-1}(\a_n+\a_{n-1}+\a_{n-2})\\
 c_{n,n-3} &= \b_n\b_{n-1}\b_{n-2}.
 \end{align}
 \end{subequations}
 We now write \eqref{quasi} as
 \begin{align}\label{AQquasi}
 x\deriv{P_n}{x}(x) = c_{n,n-3} P_{n-3}(x) + c_{n,n-2} P_{n-2}(x) +c_{n,n-1}P_{n-1}+ c_{n,n} P_{n}(x).
 \end{align} Iterating \eqref{eq:3trr} to express $P_{n-3}$ and $P_{n-2}$ in terms of $P_n$ and $P_{n-1}$, we obtain
 \begin{subequations}\label{recoa}\begin{align}
 P_{n-2}(x)&= \dfrac{x-\a_{n-1}}{\b_{n-1}}\,P_{n-1}(x)-\dfrac{P_n(x)}{\b_{n-1}}\\
 P_{n-3}(x)&
 = \left\{\dfrac{(x-\a_{n-1})(x-\a_{n-2})}{\b_{n-1}\b_{n-2}}-\dfrac{1}{\b_{n-2}}\right\} P_{n-1}(x) - \dfrac{x-\a_{n-2}}{\b_{n-1}\b_{n-2}}\,P_{n}(x). 
 \end{align}\end{subequations}
 Substituting \eqref{eq:cnn4}, \eqref{Aau} and \eqref{recoa} into \eqref{AQquasi} yields
 \begin{align*}
 x\deriv{P_n}{x}(x)=&\b_n\left\{x^2+\a_nx+\a_n^2+\b_n+\b_{n+1}-t\right\}P_{n-1}(x)\\&
 -\left\{\b_n x-\a_n \b_{n+1}+\tfrac{1}{2} \left(\la -\a_{n+1} \b_{n+1}+\a_{n-1} \b_n-\a_n^3+t \a_n+1\right)\right\}P_n(x)
 \end{align*} and hence
 \begin{align*}
 \deriv{P_n}{x}(x)= \b_n A_n(x) P_{n-1}(x)- B_n(x) P_n(x)
 \end{align*}
 where $ A_n(x)$ and $ B_n(x)$ are given by \eqref{AnBna}.
 \end{proof}
A differential equation satisfied by generalised Airy polynomials can be obtained by differentiating the differential-difference equation \eqref{ddee}.
\begin{theorem}For the generalised Airy weight \eqref{genAiry} the monic orthogonal polynomials $P_n(x;t,\la)$ with respect to this weight satisfy the differential equation
\[
\deriv[2]{P_n}{x}(x;t,\la)+\mathcal{Q}_n(x)\deriv{P_n}{x}(x;t,\la)+\mathcal{T}_n(x)P_n(x;t,\la)=0
\]
where
\begin{align*}
\mathcal{Q}_n(x)=&\frac{\lambda +t x-x^3+1}{x}-\frac{\alpha _n+2 x}{\mathcal{C}_n(x)}\\
\mathcal{T}_n(x)=&\frac{n-\left(\alpha _{n-1}+\alpha _n\right) \beta _n-\left(\beta _n \mathcal{D}_n(x)-n\right) \left(-\lambda +\beta _n \mathcal{D}_n(x)-n-t x+x^3\right)+\beta _n \mathcal{C}_{n-1}(x) \mathcal{C}_{n}(x)}{x^2}\\&+\frac{\left(n-\beta _n \mathcal{D}_n(x)\right) \left(x^2-\alpha _n^2-\beta _n-\beta _{n+1}+t\right)}{x^2\mathcal{C}_{n} (x)}\end{align*} with\begin{align*} \mathcal{C}_n(x)=&x^2+\beta _n+\beta _{n+1}+\alpha _n \left(\alpha _n+x\right)-t\\\mathcal{D}_n(x)=&\alpha _{n-1}+\alpha _n+x\end{align*}
and $\a_n(t;\la)$ and $\b_n(t;\la)$ are the coefficients in the three-term recurrence relation \eqref{eq:scrr}.
\end{theorem}
\begin{proof}For the weight \eqref{genAiry}, we have that \[ \v(x) =-\ln\w(x)=\tfrac13x^3-tx-\la\ln x.\]
The result follows by substituting the expressions for $\v(x)$ and $A_n(x)$ and $B_n(x)$, given in \eqref{AnBna},  
into the equations, see equations (3.2.13) and (3.2.14) in \cite{refIsmail},\begin{align*}
\mathcal{Q}_n(x)&=-\v'(x)-\frac{A_n'(x)}{A_n(x)}\\
 \mathcal{T}_n(x)&=B_n'(x)-B_n(x)\frac{A_n'(x)}{A_n(x)}-B_n(x)[\v'(x)+B_n(x)]+{\b_{n}}{A_{n-1}(x)A_n(x)}.
 \end{align*} Note that here equation (3.2.14) in \cite{refIsmail} has been written for monic polynomials. 
The expression $r_n=(\a_n+\a_{n-1})\b_n-n$ is also used, see \eqref{eq4b} below. \end{proof}
Next, we consider a mixed recurrence relation connecting generalised Airy polynomials associated with different weight functions. Mixed recurrence relations such as these are typically used to prove interlacing and Stieltjes interlacing of the zeros of two polynomials from different sequences and also provide a set of points that can be applied as inner bounds for the extreme zeros of polynomials. 
\begin{lemma}\label{mreca} Let $\{P_n(x;t,\la)\}_{n=0}^{\infty}$ be the sequence of monic generalised Airy polynomials orthogonal with respect to the weight \eqref{genAiry}, then, for $n$ fixed, 
\begin{align}
 \label{l+2again}x^2P_{n-2}(x;t,\la+2)&=\left[\frac{e_{n}}{\b_{n-1}}(x-\a_{n-1})-d_{n}\right]P_{n-1}(x;t,\la)+\left(1-\frac{e_{n}}{\b_{n-1}}\right)P_n(x;t,\la)
\end{align} where 
\begin{align*}d_{n}=&\dfrac{P_{n}(0;t,\la)}{P_{n-1}(0;t,\la)}+\dfrac{P_{n-1}(0;t,\la+1)}{P_{n-2}(0;t,\la+1)},\qquad
e_{n}=\dfrac{P_{n-1}(0;t,\la+1)}{P_{n-2}(0;t,\la+1)}\dfrac{P_{n-1}(0;t,\la)}{P_{n-2}(0;t,\la)}\end{align*} and $\a_n(t;\la)$ and $\b_n(t;\la)$ are the coefficients in the three-term recurrence relation \eqref{eq:scrr}.
\end{lemma}

\begin{proof}The weight function associated with the monic polynomials $P_n(x;t,\la+2)$ is 
\begin{align*}\w(x;t,\la+2)&=x^{\la+2}\exp\left(-\tfrac13x^3+tx\right) 
=x\,\w(x;t,\la+1).\end{align*}
Applying Christoffel's formula (cf.~\cite[Theorem 2.5]{refSzego}, \cite[Theorem 2.7.1]{refIsmail}) to the monic polynomial\\$P_{n-2}(x;t,\la+2)$, we can write
 \begin{align*}
x P_{n-2}(x;t,\la+2)&=\dfrac{-1}{P_{n-2}(0;t,\la+1)}\left|\begin{matrix}P_{n-2}(x;t,\la+1)&P_{n-1}(x;t,\la+1)\\P_{n-2}(0;t,
\la+1)&P_{n-1}(0;t,\la+1)\end{matrix}\right|.\end{align*}This yields
\begin{align*}
P_{n-2}(x;t,\la+2)=&
\dfrac{1}{x}\left[P_{n-1}(x;t,\la+1)-\dfrac{P_{n-1}(0;t,\la+1)}{P_{n-2}(0;t,\la+1)}P_{n-2}(x;t,\la+1)\right]\\=&
\dfrac{1}{x}\left\{\dfrac{1}{x}\left[P_{n}(x;t,\la)-\dfrac{P_{n}(0;t,\la)}{P_{n-1}(0;t,\la)}P_{n-1}(x;t,\la)\right]\right.\\&\qquad\left.-\dfrac{P_{n-1}(0;t,\la+1)}{xP_{n-2}(0;t,\la+1)}\left[P_{n-1}(x;t,\la)-\dfrac{P_{n-1}(0;t,\la)}{P_{n-2}(0;t,\la)}P_{n-2}(x;t,\la)\right]\right\}
\\=&\dfrac{1}{x^2}\left[P_{n}(x;t,\la)-d_nP_{n-1}(x;t,\la)+e_nP_{n-2}(x;t,\la)\right].
\end{align*}
Using the three-term recurrence relation \[P_{n-2}(x;t,\la)=\frac{x-\a_{n-1}}{\b_{n-1}}\, P_{n-1}(x;t,\la)-\frac{1}{\b_{n-1}}P_n(x;t,\la)\] to eliminate $P_{n-2}(x;t,\la)$ yields the result. 
\end{proof}

\comment{\begin{remark} Chihara (cf.~\cite[p. 37]{refChihara}) showed that $$P_n(x;t,\la+1)=\dfrac{h_n}{P_n(0;t,\la)}K_n(0,x)$$ where 
 $$K_n(y,x)=\sum_{j=0}^n\dfrac{P_j(y;t,\la)P_j(x;t,\la)}{h_j}$$ are the monic Kernel polynomials associated with $P_n(x;t,\la)$ and \[h_j=\imp P_j(x;t,\la)^2\d\w(x;t,\la).\] It follows that 
\[e_{n+2}=\frac{h_{n+1}}{h_n}\frac{K_{n+1}(0,x)}{K_n(0,x)}>0\] and hence
\beq\label{coef}e_{n+2}-\b_{n+1}=\frac{h_{n+1}}{h_n}\left(\frac{K_{n+1}(0,x)}{K_n(0,x)}-1\right)>0.
\eeq 
\end{remark}}

\subsection{Zeros of generalised Airy polynomials}\label{sec:GenAiryzeros}
 The property that the zeros $x_{k,n}$, $k\in\{1,2,\dots,n\}$ of $P_n(x)$ where $\{P_n(x)\}_{n=0}^{\infty}$ is a sequence of polynomials orthogonal with respect to a semiclassical weight $\w(x)>0$ are real, distinct and \[x_{1,n} <x_{1,n-1} <x_{2,n} <\dots<x_{n-1,n} <x_{n-1,n-1}<x_{n,n} \label{sep}\] holds for any semiclassical weight. The method of proof (see, for example, \cite[Theorems 3.3.1 and 3.3.2]{refSzego}) uses the three-term recurrence relation and definition of orthogonality. Note that, for an even weight $\w(x)$, $x_{\frac{n+1}{2},n} =0$ when $n$ is odd.


 Monotonicity of the zeros of semiclassical orthogonal polynomials plays and important role in applications. 
\begin{lemma}\label{mono}
Consider the semiclassical weight \beq \label{scweight1}\w(x;t,\la)=|C(x)|^{\la}\w_0(x)\exp\{tD(x)\},\qquad \qquad\la>-1\eeq 
where $\w_0(x)$ is a positive function on $(a,b)$. Let $\{P_n(x;t,\la)\}_{n=0}^{\infty}$ be the sequence of semiclassical orthogonal polynomials  associated with the weight \eqref{scweight1}. Denote the $n$ real zeros of $P_n(x;t,\la)$ in increasing order by $x_{n,\nu}(t;\la)$, $\nu=1,2,\dots,n$. Then, for a fixed value of $\nu$, $\nu\in \{1,2,\dots, n\}$, the $\nu$-th zero $x_{n,\nu}(t;\la)$
\begin{itemize}
\item[(i)] increases when $\la$ increases, if $\displaystyle{\frac{1}{C(x)}\deriv{}{x}C(x)>0}$ for $x\in (a,b)$;\\[-0.4cm]
\item[(ii)] increases when $t$ increases, if $\displaystyle{\deriv{}{x}D(x)>0}$ for $x\in (a,b)$.
\end{itemize}
\end{lemma}
\begin{proof}
\begin{itemize}\item[]
\item[(i)]For the semi-classical weight \eqref{scweight1}
\begin{align*}\frac{\partial}{\partial \la}\ln \w(x;t,\la)&= \dfrac{|C(x)|^\la \exp\{tD(x)\} w_0(x)\ln |C(x)|}{|C(x)|^\la \exp\{tD(x)\}w_0(x)} 
=\ln|C(x)|\end{align*} and therefore it follows from from Markov's monotonicity theorem (cf.~\cite[Theorem 6.12.1]{refSzego} that the zeros of $P_n(x)$ increase as $\la$ increases when $\ln|C(x)|$ is an increasing function of $x$. Since
\[\ds{\deriv{}{x}\ln|C(x)|=\dfrac{1}{|C(x)|} \text{sgn}\, C(x) \deriv[]{}{x}C(x)}\] the result follows.
\item[(ii)]Similarly, since 
\begin{align*}\frac{\partial}{\partial t}\ln\w(x;t,\la)&=\dfrac{D(x)|C(x)|^\la \exp\{tD(x)\} w_0(x)}{|C(x)|^\la \exp\{tD(x)\}w_0(x)} 
=D(x) \end{align*} 
it follows that the zeros of $P_n(x)$ increase as $t$ increases when $D(x)$ is an increasing function of $x$.\end{itemize}
\end{proof}
\begin{corollary}Let $\{P_n(x)\}_{n=0}^{\infty}$ be the sequence of monic generalised Airy polynomials orthogonal with respect to the weight \eqref{genAiry} and let $0<x_{n,n}<\dots<x_{2,n} <x_{1,n} $ denote the zeros of $P_n(x)$. Then, for $\la>-1$ and $t\in \R$ and for a fixed value of $\nu$, $\nu\in \{1,2,\dots, n\}$, 
the $\nu$-th zero $x_{n,\nu} $ increases when (i) $\la$ increases; and (ii)
$t$ increases.
\end{corollary} 
\begin{proof} This follows from Lemma \ref{mono}, taking $C(x)=x$, $D(x)=x$ and $\w_0(x)=\exp(-\tfrac13x^3)$.
\end{proof}
\comment{\begin{theorem}
Let $\{p_n\}_{n=0}^\infty$ be a sequence of polynomials orthogonal on the interval $(c,d)$.
Fix $k,n \in\mathbb{N}$ with $k < n-1$ and suppose deg$(g_{n-k-1})=n-k-1$ with
\beq\label{13}f(x)g_{n-k-1}(x)=G_k(x) p_{n-1}(x) + H(x) p_n(x)\eeq
where $f(x) \neq 0$ for $x\in(c,d)$ and deg($G_k)=k$. Then, if $g_{n-k-1}$ and $p_n$ are co-prime, 
\begin{itemize}
\item[(i)] the $n-1$ real, simple zeros of $G_k g_{n-k-1}$ interlace with the zeros of $p_n$.
\item[(ii)] The largest (smallest) zero of $G_k$ is a strict
lower (upper) bound for the largest (smallest) zero of $p_{n}.$
\end{itemize}

\end{theorem}
\begin{proof}
See \cite{refDJ12}.
\end{proof}}
Next we use \eqref{l+2again} to obtain an upper bound for the smallest zero and a lower bound for the largest zero of generalised Airy polynomials. 
\begin{theorem}Let $\{P_n(x;t,\la)\}_{n=0}^{\infty}$ be the sequence of monic generalised Airy polynomials orthogonal with respect to the weight \eqref{genAiry} on $(0,\infty)$. For each $n=2,3,\dots,$ the largest zero, $x_{1,n}$, and the smallest zero $x_{n,n}$ of $P_n(x;t,\la)$, satisfies \[0<x_{n,n}<\a_{n-1}+\frac{d_{n}\b_{n-1}}{e_{n}}<x_{1,n}\] where $\a_n=\a_n(t;\la)$ and $\b_n=\b_n(t;\la)$ are the coefficients in the three-term recurrence relation \eqref{eq:scrr} and \begin{align*}d_{n}=&\dfrac{P_{n}(0;t,\la)}{P_{n-1}(0;t,\la)}+\dfrac{P_{n-1}(0;t,\la+1)}{P_{n-2}(0;t,\la+1)},\qquad
	e_{n}=\dfrac{P_{n-1}(0;t,\la+1)}{P_{n-2}(0;t,\la+1)}\dfrac{P_{n-1}(0;t,\la)}{P_{n-2}(0;t,\la)}.\end{align*}
\end{theorem}
\begin{proof}Let \[x_{n,n}<x_{n-1,n}<\dots<x_{2,n}<x_{1,n}\] denote the zeros of $P_n(x;t,\la)$.
Consider \eqref{l+2again}\beq\label{l+2again!}x^2P_{n-2}(x;t,\la+2)=G(x)P_{n-1}(x;t,\la)+\left(1-\frac{e_{n}}{\b_{n-1}}\right)P_n(x;t,\la)\eeq with $\ds{G(x)=\frac{e_{n}}{\b_{n-1}}(x-\a_{n-1})-d_{n}}$. Since $P_{n-1}(x;t,\la)$ and $P_{n}(x;t,\la)$ are always
co-prime while $P_{n}(x;t,\la)$ and \newline $P_{n-2}(x;t,\la+2)$ are co-prime by assumption, it follows from \eqref{l+2again!} that $G(x_{j,n})\neq 0$ for every $j\in\{1,2,\dots,n\}.$ From \eqref{l+2again!}, provided $P_n(x;t,\la)\neq 0$, we have \[ \frac{x^2
P_{n-2}(x;t,\la+2)}{P_n(x;t,\la)} = 1-\frac{e_{n}}{\b_{n-1}}
+\frac{G(x) P_{n-1}(x;t,\la)}{ P_{n}(x;t,\la)}.\] The decomposition into partial fractions (cf.~\cite[Theorem 3.3.5]{refSzego} \[\frac{P_{n-1}(x;t,\la)}){P_{n}(x;t,\la)}=\sum_{j=1}^{n}\frac{C_j}{x-x_{j,n}}\] where $C_j>0$ for every $j\in\{1,2,\dots,n\}$, implies that we can write 
\[\frac{x^2
P_{n-2}(x;t,\la+2)}{P_{n}(x;t,\la)} = 1-\frac{e_{n}}{\b_{n-1}}
+\sum_{j=1}^{n}\frac{G(x)C_j}{x-x_{j,n}},~~~ x\neq x_{j,n}.\] 
Suppose that $G(x)$ does not change sign in an interval $(x_{j+1,n},x_{j,n})$ where $j\in\{1,2,\dots,n-1\}$. Since
$C_j>0$ while the right hand side takes arbitrarily large positive and
negative values on $(x_{j+1,n},x_{j,n})$, it follows that $P_{n-2}(x;t,\la+2)$ must have an odd number of zeros in every interval in which $G(x)$ does not change
sign. Since $G(x)$ is of degree $1$, there are at least $n-2$ intervals $(x_{j+1,n},x_{j,n})$, $j\in\{1,2,\dots,n-1\}$ in which $G(x)$ does not
change sign and so each of these intervals must contain exactly one of the $n-2$ real, simple zeros of $P_{n-2}(x;t,\la+2)$. We deduce that the
zero of $G(x)$, together with the $n-2$ zeros of $P_{n-2}(x;t,\la+2)$, interlaces with the $n$ zeros of $P_{n}(x;t,\la)$ and therefore the zero $\ds{\a_{n-1}+\frac{d_{n}\b_{n-1}}{e_{n}}}$ of $G(x)$ has to lie between the two extreme zeros of $P_n(x;t,\la)$. 
\end{proof}
\subsection{Differential and discrete equations satisfied by the recurrence coefficients}
In this section we discuss properties for the recurrence coefficients $\a_n(t;\la)$ and $\b_n(t;\la)$ in the three-term recurrence relation \eqref{eq:scrr}.

\begin{theorem}{
If $P_n(x)$ are the monic orthogonal polynomials for the weight $\w(x) = \e^{-\v(x)}$, then
\begin{align*}
&\left(\deriv{}{x}+B_n(x)\right) P_n(x)=\b_nA_n(x)P_{n-1}(z)\\
&\left(\deriv{}{x}-B_n(x)-\v'(x)\right) P_{n-1}(x)=-A_{n-1}(x)P_n(x),
\end{align*}
where the functions $A_n(x)$ and $B_n(x)$ are given by
\begin{align*}
A_n(x) &= \frac{1}{h_n}\int_0^\infty \frac{\v'(x)-\v'(y)}{x-y} \, P_n^2(y)\, \w(y)\,\d y\\
B_n(x) &= \frac{1}{h_{n-1}}\int_0^\infty \frac{\v'(x)-\v'(y)}{x-y} \, P_n(y)P_{n-1}(y) \,\w(y)\,\d y.
\end{align*}
}\end{theorem}

The functions $A_n(x)$ and $B_n(x)$ also 
satisfy the following supplementary conditions.
\begin{lemma}{The functions $A_n(x)$ and $B_n(x)$ satisfy
\begin{align}
& B_{n+1}(x)+B_n(x) = (x-\a_n)A_n(x)-\v'(x)\label{eq1}\\
& 1+(x-\a_n)[B_{n+1}(x)-B_n(x)] = \b_{n+1}A_{n+1}(x)-\b_nA_{n-1}(x)\label{eq2}\\
& B_n^2(x)+\v'(x)B_n(x)+\sum_{j=0}^{n-1}A_j(x)=\b_nA_n(x)A_{n-1}(x).\label{eq3}
\end{align}
}\end{lemma}
\begin{proof} See \cite[Lemma 3.2.2, Theorem 3.2.4]{refIsmail}; see also \cite[Proposition 3.1]{refFVAZ}. \end{proof}

\begin{theorem}\label{thm:anbn1}
The recurrence coefficients $\a_n(t;\la)$ and $\b_n(t;\la)$ for the generalised Airy weight \eqref{genAiry} satisfy the discrete system
\begin{subequations}\label{anbn:dissys}\begin{align} \label{anbn:dissysa}
& (2\a_n+\a_{n-1})\b_n + (\a_{n+1}+2\a_{n})\b_{n+1}+\a_n^3-t\a_n =2n+\la+1\\
& \b_n^3 + (\b_{n+1}+\b_{n-1}-2\a_n\a_{n-1}-2t)\b_n^2 \nonumber\\ &\qquad + \{ (\b_{n+1}+\a_n^2-t)(\b_{n-1}+\a_{n-1}^2-t) +(\a_n+\a_{n-1})(2n+\la) \}\b_n = n(n+\la).
\end{align}\end{subequations}
\end{theorem}
\begin{proof} 
Substituting \eqref{AnBna} into \eqref{eq1}, with $v(x)=\tfrac13x^3-tx-\la\ln(x)$, then from the coefficients of $x^0$ and $x^{-1}$ we obtain
\begin{align}
&\b_n+\b_{n+1}=R_n-\a_n^2 + t \label{eq1a}\\
& r_n+r_{n+1} =-\a_n R_n + \la\label{eq1b}
\end{align}
\comment{Similarly from \eqref{eq3} we obtain
\begin{align}
&\a_n(\b_n-\b_{n+1}) + r_{n+1}-r_n + 1 = \a_{n+1}\b_{n+1}-\a_{n-1}\b_n\label{eq2a}\\
&\a_n(r_n-r_{n+1}) =R_{n+1}\b_{n+1}-R_{n-1}\b_n .\label{eq2b}
\end{align}}%
and substituting \eqref{AnBna} into \eqref{eq2}, then from the coefficients of $x$ and $x^{-2}$ we obtain
\begin{align}
&r_n+n=(\a_n+\a_{n-1})\b_n\label{eq3a}\\
&r_n^2-\la r_n = \b_n R_n R_{n-1}.\label{eq3b}
\end{align}
From \eqref{eq1a} and \eqref{eq3a}, respectively, we see that 
\begin{align} &R_n = \b_n+\b_{n+1}+\a_n^2-t\label{eq4a}\\ &r_n= (\a_n+\a_{n-1})\b_n-n.\label{eq4b}\end{align}
Substituting these into \eqref{eq1b} and \eqref{eq3b} gives the discrete system \eqref{anbn:dissys}, as required.
 \end{proof}
\begin{corollary}
{For the generalised Airy weight \eqref{genAiry}
the monic orthogonal polynomials $P_{n}(x;t,\la)$ with respect to this weight satisfy the differential-difference equation \eqref{ddee}
with
\begin{align*}
A_n(x)&= x + \a_n + \frac{R_n}{x},\qquad
B_n(x)=\b_n + \frac{r_n}{x}
\end{align*}
where $\a_n(t;\la)$ and $\b_n(t;\la)$ are the coefficients in the three-term recurrence relation \eqref{eq:scrr} and 
\begin{align*}R_n& = \b_n+\b_{n+1}+\a_n^2-t,\qquad r_n = (\a_n+\a_{n-1})\b_n-n.\end{align*}}
\end{corollary}
\begin{proof} The result follows from Theorem \ref{thm:dde} by substituting \eqref{anbn:dissysa} into \eqref{def:Rnrnb}.\end{proof}

\begin{theorem}\label{thm:anbn2}
The recurrence coefficients $\a_n(t;\la)$ and $\b_n(t;\la)$ for the generalised Airy weight \eqref{genAiry} satisfy the differential system
\begin{subequations}\label{anbn:diffsys}\begin{align} 
&\deriv[2]{\a_n}{t}+3\a_n\deriv{\a_n}{t}+\a_n^3+(6\b_n-t)\a_n = 2n+\la+1\\
&\left(\deriv{\a_n}{t}+\a_n^2+2\b_n-t\right)\deriv[2]{\b_n}{t}-\left(\deriv{\b_n}{t}\right)^2-\left(2\a_n\deriv{\a_n}{t}+2\a_n^3-2t\a_n+2n+\la\right) \deriv{\b_n}{t}\nonumber\\
&\qquad -\b_n \left(\deriv{\a_n}{t}\right)^2 + 4\b_n^3-4t\b_n^2+\{\a_n^4-2t\a_n^2+2(2n+\la)\a_n+t^2\}\b_n = n(n+\la).
\end{align}\end{subequations}
\end{theorem}
\begin{proof} Recall that from Theorem \ref{thm:todasys}, $\a_n$ and $\b_n$ satisfy the Toda system
\beq\label{eq:toda2}
\deriv{\a_n}t=\b_{n+1}-\b_n,\qquad\deriv{\b_n}t=\b_n(\a_n-\a_{n-1}).\eeq
We can use these to eliminate $\b_{n+1}$, $\b_{n-1}$, $\a_{n+1}$ and $\a_{n-1}$ from the discrete system \eqref{anbn:dissys}. Substituting
\begin{align*}&\b_{n+1}=\b_n+\deriv{\a_n}t, &&
\a_{n+1}= \a_n +\deriv{}t \ln\b_{n+1} =\a_n +\deriv{}t \ln\left(\b_n+\deriv{\a_n}t\right)\\ 
&\a_{n-1} = \a_n-\deriv{}t \ln\b_n, && \b_{n-1} = \b_n-\deriv{\a_{n-1}}t = \b_n-\deriv{}t \a_{n}+\deriv[2]{}t \ln\b_n,
\end{align*}
into \eqref{anbn:dissys} gives the differential system \eqref{anbn:diffsys}, as required.\end{proof}
\comment{ \beq \deriv[2]{\a_n}{t}+3\a_n\deriv{\a_n}{t}+\a_n^3+(6\b_n-t)\a_n = 2n+\la+1.
\label{eq4ode}\eeq
\comment{Using \eqref{eq:toda2} to eliminate $\b_{n+2}$, $\b_{n+1}$, $\a_{n+1}$ and $\a_{n-1}$ from \eqref{eq6} yields
\[ \deriv[3]{\a_n}{t}+3\a_n\deriv[2]{\a_n}{t}+3\left(\deriv{\a_n}{t}\right)^2+(3\a_n^2+6\b_n-t)\deriv{\a_n}{t} + 6\a_n\deriv{\b_n}{t} -\a_n =0\]
which is the differential of \eqref{eq4ode} w.r.t.\ $t$. }%
Using \eqref{eq:toda2} to eliminate $\b_{n+1}$, $\a_{n+1}$ and $\a_{n-1}$ from \eqref{eq6} yields
\begin{align}
&\left(\deriv{\a_n}{t}+\a_n^2+2\b_n-t\right)\deriv[2]{\b_n}{t}-\left(\deriv{\b_n}{t}\right)^2-\left(2\a_n\deriv{\a_n}{t}+2\a_n^3-2t\a_n+2n+\la\right) \deriv{\a_n}{t}\nonumber\\
&\qquad -\b_n \left(\deriv{\a_n}{t}\right)^2 + 4\b_n^3-4t\b_n^2+\{\a_n^4-2t\a_n^2+2(n+\la)\a_n+t^2\}\b_n = n(n+\la).\label{eq5ode}
\end{align}}


\subsection{Asymptotics of the recurrence coefficients.}\label{sec:rc}

\begin{lemma}{As $n\to\infty$, the recurrence coefficients $\a_n(t;\la)$ and $\b_n(t;\la)$ have the formal asymptotic expansions
\begin{subequations}\label{nasym:anbn}\begin{align}
\a_n(t;\la)&= \frac{2n^{1/3}}{\k} + \frac{\k t}{15n^{1/3}}+\frac{\k^{2}(\la+1)}{30\,n^{2/3}} + \O(n^{-1})\\
\b_n(t;\la)&= \frac{n^{2/3}}{\k^2} + \frac{t}{15}+\frac{\k \la}{30\,n^{1/3}}+\frac{\k^{2}t^2}{900\,n^{2/3}} + \O(n^{-1}),
\end{align}\end{subequations}
where $\k=\sqrt[3]{10}$.
}\end{lemma}
\comment{\begin{align*}
\a_n(t;\la)&= 2N + \frac{t}{15N}+\frac{\la+1}{30N^2} + \O(N^{-3})\\
\b_n(t;\la)&= N^2 + \frac{t}{15}+\frac{\la}{30N}+\frac{t^2}{900N^2} + \O(N^{-3})\\
\end{align*}
where $N=\sqrt[3]{n/10}$.}
\begin{proof}
From \eqref{anbn:dissys}, it follows that as $n\to\infty$, $\a_n\sim\a n^{1/3}$ and $\b_n\sim\b n^{2/3}$, where $\a$ and $\b$ are constants that satisfy the algebraic system
\begin{align*} 
& 6\a\b+\a^3 =2,\qquad
4\b^3 +\a^4\b +4\a\b = 1,
\end{align*}
which has solution $\a=2/\k$ and $\b=1/\k^2$, with $\k=\sqrt[3]{10}$. Now we suppose that as $n\to\infty$
\begin{subequations} \label{anbn:nexp} \begin{align}
\a_n(t;\la)&= \frac{2n^{1/3}}{\k} + a_0(t) + \frac{a_1(t)}{n^{1/3}} + \frac{a_2(t)}{n^{2/3}} + \O(n^{-1})\\
\b_n(t;\la)&= \frac{n^{2/3}}{\k^2} + \tilde{b}_1(t)n^{1/3}+ b_0(t) + \frac{b_1(t)}{n^{1/3}} + \frac{b_2(t)}{n^{2/3}} + \O(n^{-1}),
\end{align}\end{subequations}
where $a_0(t)$, $a_1(t)$, $a_2(t)$, $\tilde{b}_1(t)$, $b_0(t)$, $b_1(t)$ and $b_2(t)$ are to be determined. Substituting \eqref{anbn:nexp} into the system \eqref{anbn:diffsys}, equating coefficients of powers of $n$ and solving the resulting system gives
\[ \begin{split} &a_0(t) = 0, \qquad a_1(t)=\frac{\k t}{15},\qquad a_2(t)= \frac{\k^2(\la+1)}{30}\\
& \tilde{b}_1(t)= 0, \qquad b_0(t)=\frac{t}{15},\qquad b_1(t) = \frac{\k \la}{30},\qquad b_2(t)=\frac{\k^{2}t^2}{900},
\end{split} \]
and so we obtain \eqref{nasym:anbn}, as required.
\end{proof}

\begin{lemma}{\label{lem:anbn:asympt} As $t\to\infty$, the recurrence coefficients $\a_n(t;\la)$ and $\b_n(t;\la)$ for the generalised Airy weight \eqref{genAiry}
have the formal asymptotic expansions
\begin{subequations}\label{asym:anbntp}\begin{align}
\a_n(t;\la)&= \sqrt{t}- \frac{2n-2\la+1}{4t} + \O(t^{-5/2})\label{asym:antp}\\
\b_n(t;\la) &= \frac{n}{2\sqrt{t}} + \frac{n(n-2\la)}{4t^2} + \O(t^{-7/2}).\label{asym:bntp}
\end{align}\end{subequations}
As $t\to-\infty$, the recurrence coefficients $\a_n(t;\la)$ and $\b_n(t;\la)$ have the formal asymptotic expansions
\begin{subequations}\label{asym:anbntm}\begin{align}
\a_n(t;\la)&= -\frac{2n+\la+1}{t}-\frac{(2n+\la+1)(10n^2+10n\la+\la^2+10n+5\la+6)}{t^4}+\O(t^{-7})\label{asym:antm}\\
\b_n(t;\la)& = \frac{n(n+\la)}{t^2} + \frac{4n(n+\la)(5n^2+5n\la+\la^2+1)}{t^5} +\O(t^{-8}).\label{asym:bntm}
\end{align}\end{subequations}
}\end{lemma}

\begin{proof} When $n=0$ we know that 
\[ \a_0(t;\la)=\deriv{}{t}\ln\mu_0(t;\la),\qquad \b_0(t;\la)=0\]
where \[\mu_0(t;\la)= \int_0^\infty x^{\la}\exp\left(-\tfrac13x^3+tx\right)\,\d x.\]
Further $\a_0(t;\la)$ satisfies the equation 
\beq \deriv[2]{\a_0}{t}+3\a_0\deriv{\a_0}{t}+\a_0^3-t\a_0 = \la+1.
\label{eq:a0ode}\eeq
Using Laplace's method it follows that as $t\to\infty$
\[\begin{split} \mu_0(t;\la) & 
= t^{(\la+1)/2}\int_0^\infty \xi^{\la}\exp\{t^{3/2}\xi(1-\tfrac13\xi^2)\}\,\d\xi \\ 
&={t}^{\la/2-1/4}\sqrt {\pi}\exp\left(\tfrac23\,{t}^{3/2}\right){\left[1+\frac{12\la^2-24\la+5}{48t^{3/2}}+\O(t^{-3})\right]} 
\end{split}\]
and so it is straightforward to show that as $t\to\infty$
\[ \a_0(t;\la)=\sqrt{t}+\O({t}^{-1}).\]
Suppose we seek an asymptotic expansion in the form
\[ \a_0(t;\la)=\sqrt{t}+\frac{a_{0,1}}{t}+\frac{a_{0,2}}{t^{5/2}}+\O(t^{-4})\]
where $a_{0,1}$ and $a_{0,2}$ are constants to be determined.
Substituting this into \eqref{eq:a0ode} and equating coefficients of powers of $t$ gives
\[ a_{0,1}=\tfrac12\la-\tfrac14,\qquad a_{0,2} = -\tfrac38\la^2+\tfrac34\la-\tfrac{5}{32}\]
and so 
\[ \a_0(t;\la)=\sqrt{t}+\frac{2\la-1}{4t}-\frac{12\la^2-24\la+5}{32t^{5/2}}+\O(t^{-4})\]
which is \eqref{asym:antp} with $n=0$.
From the Toda system \eqref{eq:toda2}, since $ \b_0=0$ then $\ds \b_1= \deriv{\a_0}{t}$,
and so
\[ \b_1(t;\la) = \frac{1}{2\sqrt{t}}-\frac{2\la-1}{4t^2} + \O(t^{-7/2})\]
which is \eqref{asym:bntp} with $n=1$.

Now we can use induction. Suppose that \eqref{asym:anbntp} are true, then using the Toda system \eqref{eq:toda2} we have
\begin{align*}\b_{n+1} &=\b_n+\deriv{\a_n}t = \frac{n+1}{2\sqrt{t}} + \frac{(n+1)(n+1-2\la)}{4t^2} + \O(t^{-7/2})\\
\a_{n+1} &= \a_n +\deriv{}t \ln\b_{n+1} = \sqrt{t}- \frac{2n-2\la+3}{4t} + \O(t^{-5/2}),
\end{align*}
which are \eqref{asym:anbntp} with $n\to n+1$, and hence the result is proved by induction.

The asymptotics \eqref{asym:anbntm} are proved in an analogous way. Using Watson's Lemma it follows that as $t\to-\infty$
\[ \mu_0(t;\la) = \frac{\Gamma(\la+1)}{(-t)^{\la+1}}\left[1+\frac{(\la+1)(\la+2)(\la+3)}{3t^3}+\O(t^{-6})\right]\]
and so as $t\to-\infty$
\[ \a_0(t;\la)=-\frac{\la+1}{t}+\O(t^{-4}).\]
In this case we seek an asymptotic expansion in the form
\[ \a_0(t;\la)=-\frac{\la+1}{t}+\frac{\widetilde{a}_{0,1}}{t^4}+\frac{\widetilde{a}_{0,2}}{t^{7}}+\O(t^{-10})\]
Substituting this into \eqref{eq:a0ode} and equating coefficients of powers of $t$ gives
\[ \widetilde{a}_{0,1}=-(\la+1)(\la+2)(\la+3),\qquad \widetilde{a}_{0,2}=-(\la+1)(\la+2)(\la+3)(3\la^2 + 21\la + 38)\]
and so 
\[ \a_0(t;\la)=-\frac{\la+1}{t}-\frac{(\la+1)(\la+2)(\la+3)}{t^4}+\O(t^{-7})\]
which is \eqref{asym:antm} with $n=0$.
From the Toda system \eqref{eq:toda2}, since $ \b_0=0$ then $\ds \b_1= \deriv{\a_0}{t}$,
and so
\[ \b_1(t;\la) =\frac{\la+1}{t^2}+\frac{4(\la+1)(\la+2)(\la+3)}{t^5}+\O(t^{-8})\]
which is \eqref{asym:bntp} with $n=1$. As for \eqref{asym:anbntp}, the results asymptotics \eqref{asym:anbntm} can be proved using induction and the Toda system \eqref{eq:toda2}.
\end{proof}

An alternative method of proving Lemma \ref{lem:anbn:asympt} is to use the differential system \eqref{anbn:diffsys} by seeking solutions in the form
\begin{align*}
\a_n(t;\la) &= \sqrt{t} +\frac{a_{n,1}}{t}+\O(t^{-5/2}),\qquad
\b_n(t;\la) = \frac{n}{2\sqrt{t}} + \frac{b_{n,1}}{t^2} +\O(t^{-7/2})
\end{align*}
as $t\to\infty$, where $a_{n,1}$ and $b_{n,1}$ are to be determined, and 
\begin{align*}
\a_n(t;\la) &= \frac{\tilde{a}_{n,0}}{t} + \frac{\tilde{a}_{n,1}}{t^4} + \O(t^{-7}),\qquad
\b_n(t;\la) = \frac{\tilde{b}_{n,0}}{t^2} + \frac{\tilde{b}_{n,1}}{t^5} + \O(t^{-8})
\end{align*}
as $t\to-\infty$, where $\tilde{a}_{n,0}$, $\tilde{a}_{n,1}$, $\tilde{b}_{n,0}$ and $\tilde{b}_{n,1}$ are to be determined. 

Plots of $\a_n(t;\la)$, and $\b_n(t;\la)$, for $n=1,2,\ldots,5$, with $\la=0,\tfrac12,2$ are given in Figure~\ref{fig:anbn}. 
\begin{figure}[ht]\[ \begin{array}{c@{\quad}c@{\quad}c}
\fig{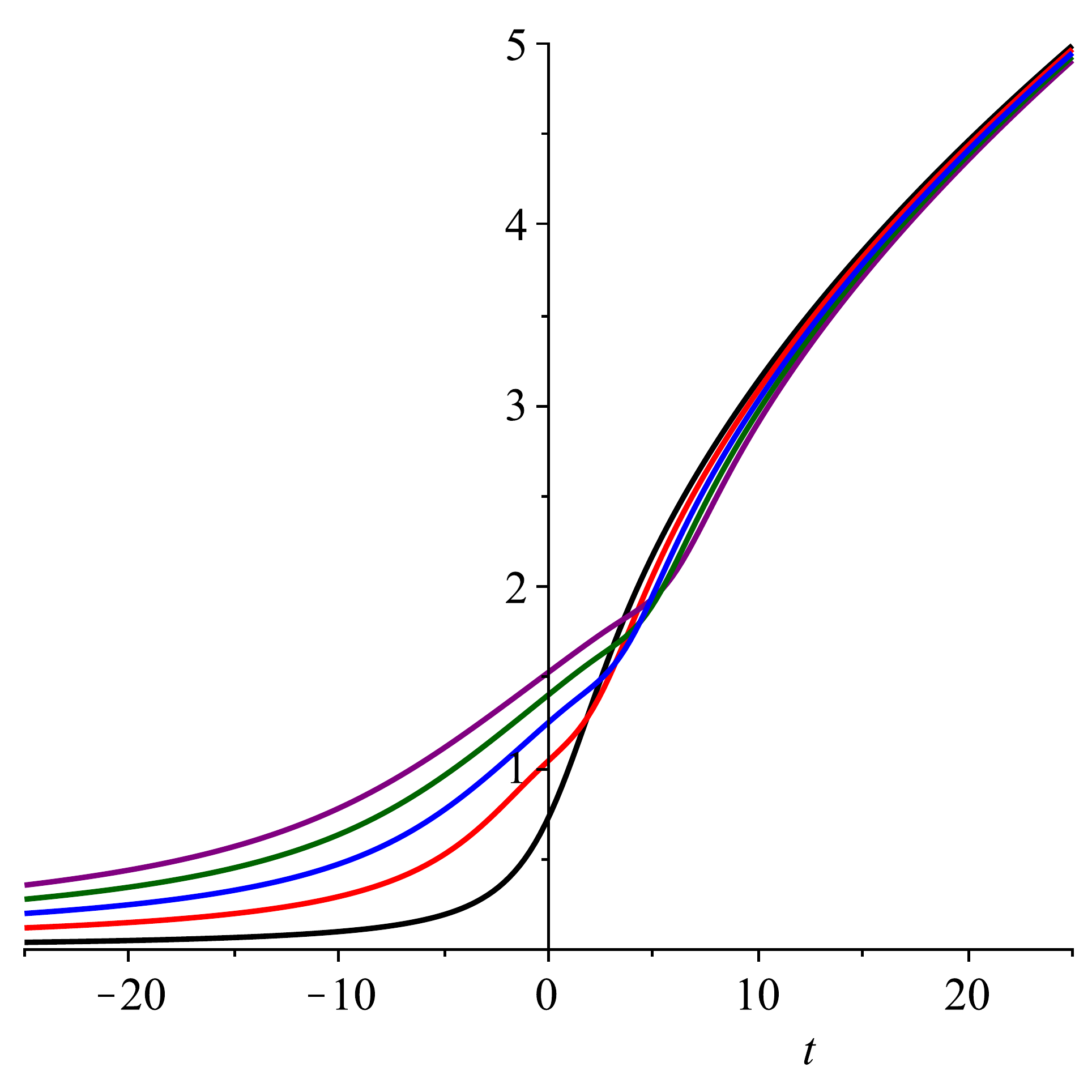} & \fig{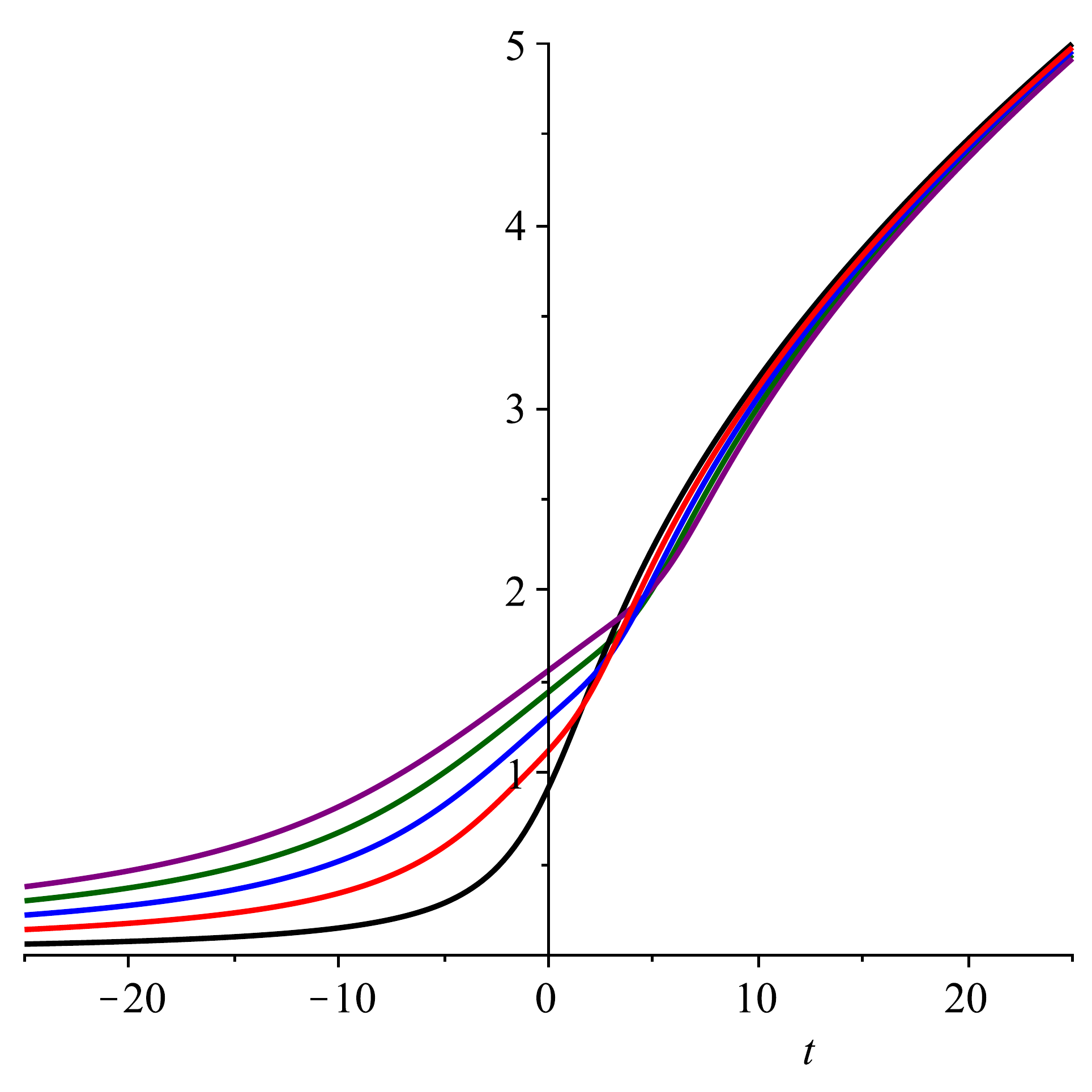} &\fig{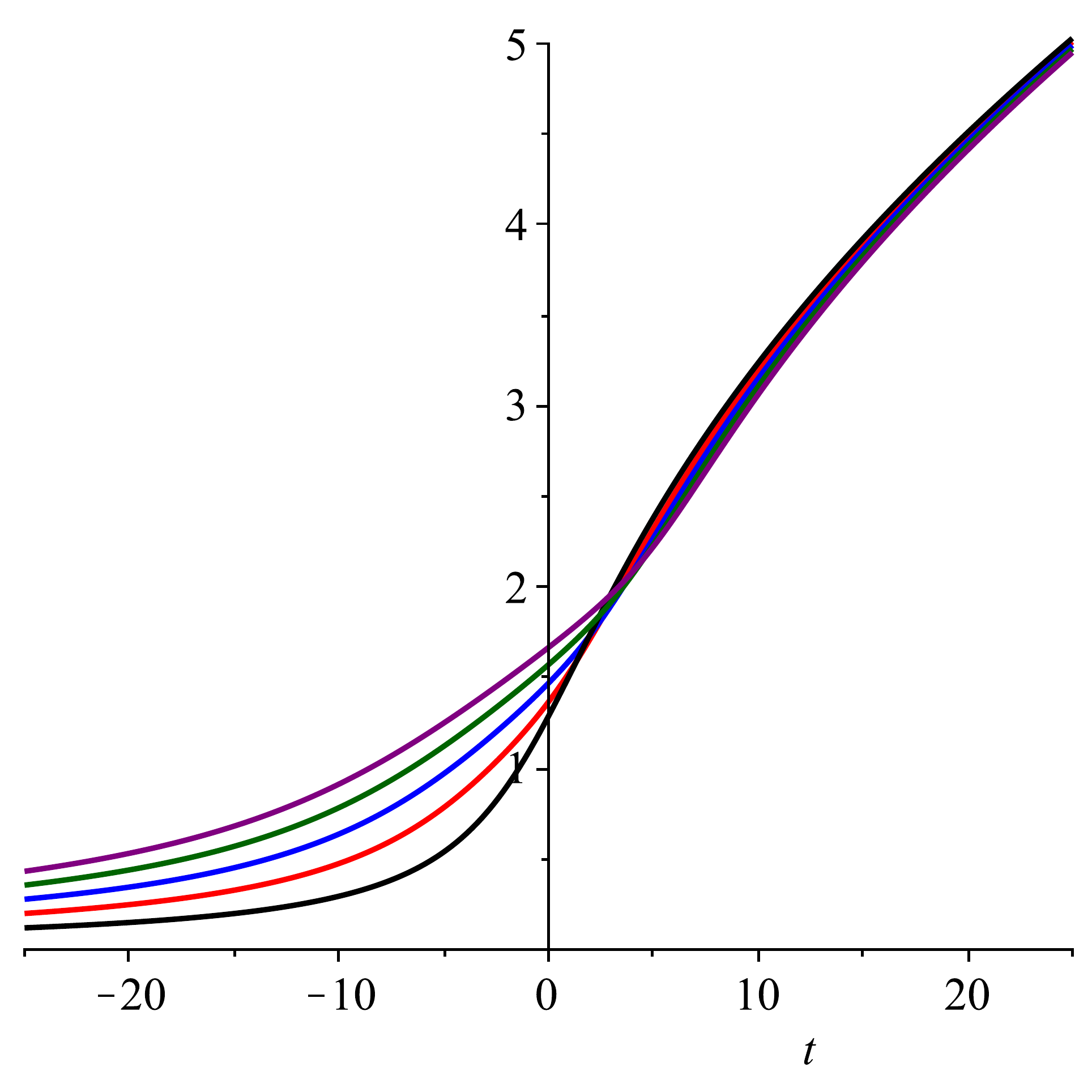}\\
\la=0 & \la=\tfrac12 & \la=2\\
\fig{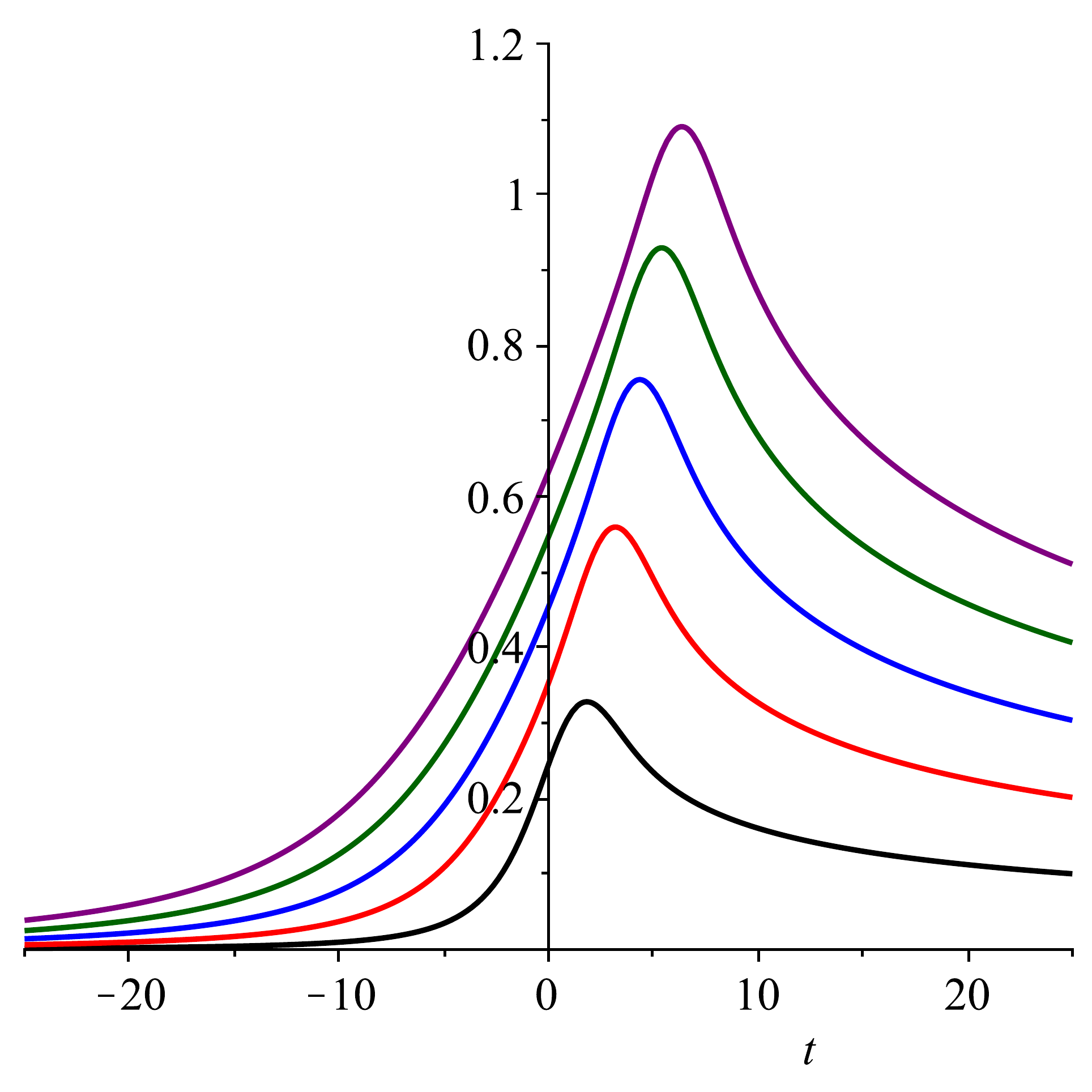} &\fig{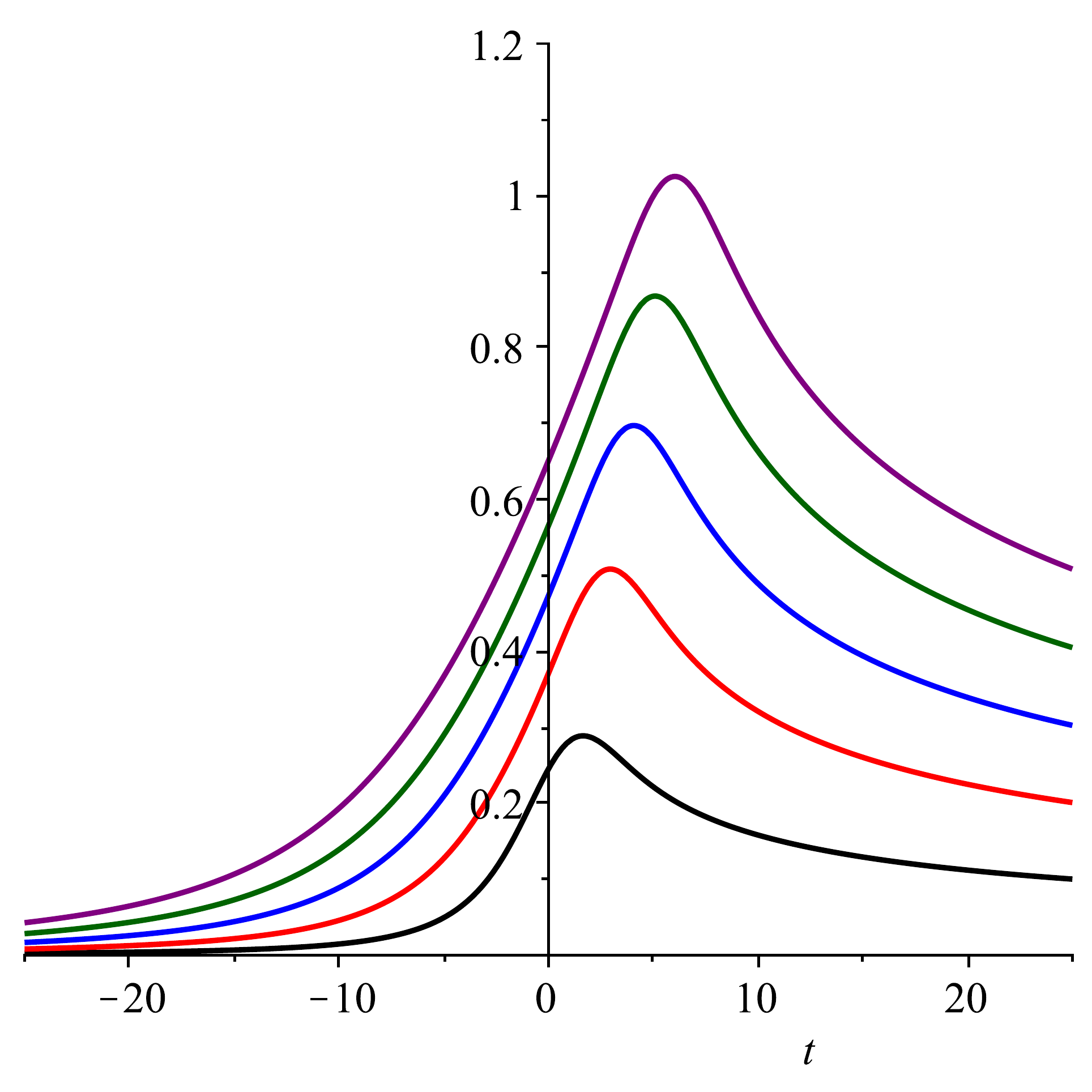} & \fig{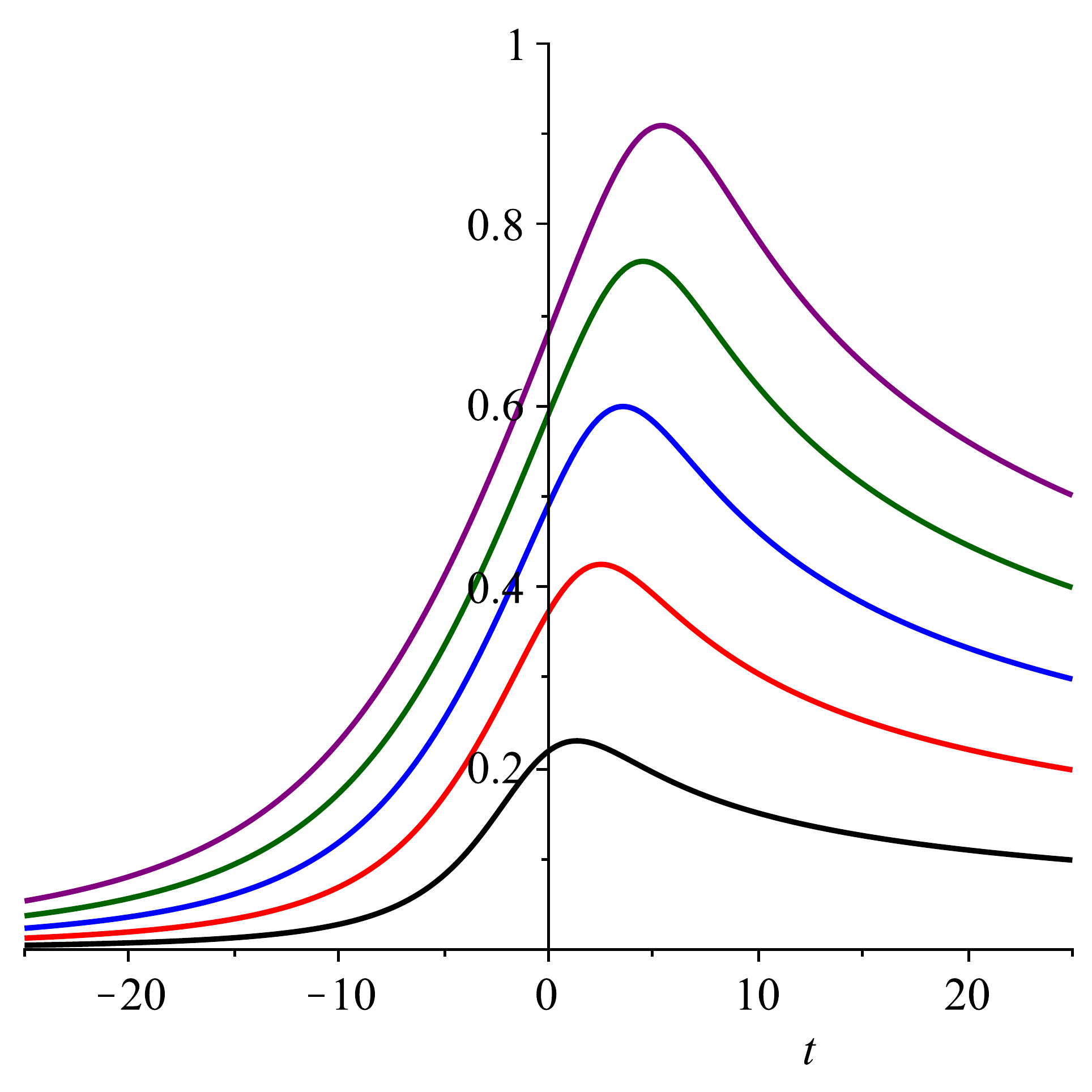}\\
\end{array}\]
\caption{\label{fig:anbn}Plots of $\a_n(t;\la)$, upper row, and $\b_n(t;\la)$, lower row, for $n=1$ (black), $n=2$ (red), $n=3$ (blue), $n=4$ (green) and $n=5$ (purple), with $\la=0,\tfrac12,2$.}
\end{figure}

From these plots we make the following conjecture.

\begin{conjecture}{\rm
\begin{enumerate}\item[]
\item The recurrence coefficient $\a_{n}(t;\la)$ is a monotonically increasing function of $t$.
\item If $\la$ is fixed, then $\b_{n+1}(t;\la)>\b_{n}(t;\la)$, for all $t$.
\item The recurrence coefficient $\b_{n}(t;\la)$ has one maximum at $t=t^*_{n}$, with $t^*_{n+1}>t^*_{n}$, with $\la$ fixed.
\end{enumerate}
}\end{conjecture}

\begin{remark}{Wang \etal\ \cite{refWZC20b} claim that the recurrence coefficients $\a_n$ and $\b_n$ for the generalised Airy weight satisfy the differential system
\beq \label{WZCeq35}
\deriv{\a_n}{t} = t-\a_n^2-2\b_n,\qquad \deriv{\b_n}{t}=2\a_n\b_n-n-\tfrac12\la
\eeq
and the discrete system
\[ (\a_n+\a_{n-1})\b_n=n+\tfrac12\la,\qquad \b_n+\b_{n+1}+\a_n^2=t \]
see equations (34), (35), (37) and (38) in \cite{refWZC20b}.
Theorems \ref{thm:anbn1} and \ref{thm:anbn2} above show that their claim is not correct.
Wang \etal\ \cite{refWZC20b} misquote the results of Magnus \cite{refMagnus95}, who considered the weight 
\[\w(x;t)=\exp\left(-\tfrac13x^3+tx\right),\qquad x\in\mathcal{C}\] where $\mathcal{C}$ is a contour in the complex plane, as being for $x\in[0,\infty)$; see equation (2) in \cite{refWZC20b}.
There are other reasons to illustrate that some results in \cite{refWZC20b} are not correct.
Eliminating $\b_n$ in \eqref{WZCeq35} gives 
\[\deriv[2]{\a_n}{t}=2\a_n^3-2t\a_n+2n+\la+1\]
which is equivalent to the second \p\ equation (\PII)\ 
\beq\label{eq:PII} \deriv[2]{q}{z}=2q^3+zq+A\eeq
with $A=n+\tfrac12(\la+1)$. 
In \cite{refWZC20b}, $\a_n$ is expressed in terms of the Hankel determinants, therefore giving special function solutions of \PII\ for \textit{all} positive values of the parameter $A$, which is not true. It is well-known that there are special function solutions of \PII\ \eqref{eq:PII} if and only if $A=n+\tfrac12$, for $n\in\Z$ \cite{refGambier09}.
\comment{Eliminating $\a_n$ gives 
\[\deriv[2]{\b_n}{t}= \frac{1}{2\b_n}\left(\deriv{\b_n}{t}\right)^2 - 4\b_n^2+2t\b_n-\frac{(2n+\la)^2}{8\b_n}\]
which is equivalent to {$\text{P}_{\!34}$}
\[\deriv[2]{p}{z}= \frac{1}{2p}\left(\deriv{p}{z}\right)^2 - 2p^2-zp-\frac{(2n+\la)^2}{2p}.\]}
}\end{remark}

\section{Generalised sextic Freud polynomials}\label{sec:gen6freud}

The generalised sextic Freud weight
\beq\label{genFreud6} \w(x;t,\la)=|x|^{2\la+1}\exp\left(-x^6+tx^2\right),\qquad x\in\R\eeq
with $t\in\R$ and $\la>-1$ parameters,
is a symmetric weight, i.e.\ $\w(x;t,\la)=\w(-x;t,\la)$, so that $\a_n \equiv0$. The generalised sextic Freud weight and recurrence coefficients associated with the weight were discussed in \cite{refCJ20}. 

Next we consider some properties of the polynomials associated with the generalised sextic Freud weight (\ref{genFreud6}). Generalised sextic Freud polynomials arise from a symmetrisation of generalised Airy polynomials. Since the weight is symmetric, the monic orthogonal polynomials $S_n(x;t,\la)$, $n\in\N$, therefore satisfy the three-term recurrence relation
\beq\label{eq:3rr}
S_{n+1}(x;t,\la)=xS_n(x;t,\la)-\b_n(t;\la)S_{n-1}(x;t,\la),\qquad n=0,1,2,\ldots\ \eeq
with $S_{-1}(x;t,\la)=0$ and $S_0(x;t,\la)=1$.
\subsection{Differential and discrete equations satisfied by generalised sextic Freud polynomials} \label{gen6feq}
In this section we derive mixed recurrence relations, differential-difference equations and differential equations satisfied by generalised sextic Freud polynomials which are analogous to those for the generalised Airy polynomials in \S\ref{sec:genairyprop}.

\comment{ \begin{theorem}\label{Thm:ABn}
Let \beq \label{gft}\ww{x}=|x-k|^\rho \exp\{-\v(x)\},\qquad x,\,t,\,k\in\R\eeq where $\v(x)$ is a continuously differentiable function on $\R$. Assume that the polynomials $\{P_n(x)\}_{n=0}^{\infty}$ satisfy the orthogonality relation
\[\imp P_n(x)P_m(x)\,\ww{x}\,\dx=h_n\delta_{mn}.\]
Then, for ${\rho\geq1}$, $P_{n}(x)$ satisfy the differential-difference equation
\beq \nonumber \label{dde}
(x-k)\,\deriv{P_n}{x}(x)=A_n(x)P_{n-1}(x)-\mathcal{B}_n(x)P_n(x)
\eeq 
where
\begin{align*}
A_n(x)&=\frac{x-k}{h_{n-1}}\imp P_n^2(y)\,\Kxy\,\w(y)\,\dy+a_n(x)\\
\mathcal{B}_n(x)&= \frac{x-k}{h_{n-1}}\imp P_n(y)P_{n-1}(y)\,\Kxy\,\w(y)\,\dy+b_n(x)\label{Bn}
\end{align*} 
with
\beq \nonumber \Kxy=\frac{\v'(x)-\v'(y)}{x-y}\label{def:Kxy}\eeq 
and
\begin{align*} a_n(x)&=\frac{\rho}{h_{n-1}}\imp \frac{P_n^2(y)}{y-k}\,\w(y)\,\dy,\qquad
b_n(x)=\frac{\rho}{h_{n-1}} \imp \frac{P_n(y)P_{n-1}(y)}{y-k}\,\w(y)\,\dy.\end{align*}
\end{theorem}

\begin{proof} See \cite[Theorem 2]{refCJK}.
\end{proof}

\begin{lemma}\label{lemmaeven}
Consider the weight defined by \eqref{gft} and assume that $\v(x)$ is an even, continuously differentiable function on $\R$. Assume that the polynomials $\{P_n(x)\}_{n=0}^{\infty}$ satisfy the orthogonality relation
\[\imp P_n(x)P_m(x)\,\ww{x}\,\dx=h_n\delta_{mn}\]
and the three-term recurrence relation
\[\label{3trr}
P_{n+1}(x)=xP_{n}(x)-\b_n(t)P_{n-1}(x)
\]
with $P_0=1$ and $P_1=x$. Then the polynomials $P_n(x)$ satisfy
\begin{align*} \imp \frac{P_n^2(y)}{y-k}\,\w(y)\,\dy&=0,\qquad
\int_{-\infty}^\infty \frac{P_n(y)P_{n-1}(y)}{y-k}\,\w(y)\,\dy= \tfrac12[1-(-1)^n]\,h_{n-1},
\label{int22}
\end{align*} where $n\in\N$ and
\[ h_n = \int_{-\infty}^\infty {P_n^2(y)\,\ww{y}}\,\dy.\]\end{lemma}

\begin{proof} See \cite[Lemma 1]{refCJK}. 

\subsubsection{The differential-difference equation satisfied by generalised sextic Freud polynomials}
\begin{lemma}\label{cor:ABn}
Let \[\ww{x}=|x|^\rho \exp\{-\v(x)\},\qquad x,\,t,\,k\in\R\] where $\v(x)$ is an even, continuously differentiable function on $\R$. Assume that the polynomials $\{P_n(x)\}_{n=0}^{\infty}$ satisfy the orthogonality relation
\[\imp P_n(x)P_m(x)\,\ww{x}\,\dx=h_n\delta_{mn}.\]
Then, for ${\rho\geq1}$, $P_{n}(x)$ satisfy the differential-difference equation
\[
x\,\deriv{P_n}{x}(x)=\mathcal{A}_n(x)P_{n-1}(x)-\mathcal{B}_n(x)P_n(x)
\]
where
\begin{align*}
\mathcal{A}_n(x)&=\frac{x}{h_{n-1}}\imp P_n^2(y)\,\Kxy\,\w(y)\,\dy\\
\mathcal{B}_n(x)&= \frac{x}{h_{n-1}}\imp P_n(y)P_{n-1}(y)\,\Kxy\,\w(y)\,\dy+\tfrac12{\rho}[1-(-1)^n].
\end{align*}
\end{lemma}
\begin{proof} 
See \cite[Corollary 1]{refCJK}.
\end{proof} 

\begin{lemma}{\label{lem31}For the generalised sextic Freud weight \eqref{genFreud6}
the monic orthogonal polynomials $P_{n}(x)$ with respect to $\ww{x}$ satisfy
\begin{subequations}
\begin{align}
\imp\Kxy&{P_n^2(y)}\,\w(y)\,\dy \nonumber\\ 
&=6\big[x^4-\tfrac13t +x^2\big(\b_n+\b_{n+1}\big)+\b_{n+2}\b_{n+1} 
+\big(\b_{n+1}+\b_{n}\big)^2+\b_{n-1}\b_{n}\big]h_n, \label{int1a}\\
\imp\Kxy& {P_n(y)P_{n-1}(y)}\,\w(y)\,\dy = 6x\big(x^2+\b_{n+1}+\b_n+\b_{n-1}\big)h_n\label{int1b}
\end{align}
\end{subequations} 
where 
\[\Kxy=\frac{\vx-\vy}{x-y}\] with
$v(x)=x^6-tx^2$ and
\beq \nonumber \label{def:hn} h_n = \imp {P_n^2(y)\,\ww{y}}\,\dy.\eeq }\end{lemma}

\begin{proof}Since
$v(x)=x^6-tx^2$, we have
\[\Kxy= 6(x^4+x^3y+x^2y^2+xy^3+y^4)-2t.\]
Hence for \eqref{int1a}
\begin{align*} \imp &\Kxy {P_n^2(y)}\,\w(y)\,\dy \\
&=
(6x^4-2t)\imp {P_n^2(y)}\,\w(y)\,\dy 
+ 6x^3\imp {yP_n^2(y)}\,\w(y)\,\dy + 6x^2\imp {y^2P_n^2(y)}\,\w(y)\,\dy\\
&\qquad + 6x\imp {y^3P_n^2(y)}\,\w(y)\,\dy + 6\imp {y^4P_n^2(y)}\,\w(y)\,\dy\\
&= (6x^4-2t)h_n +6x^2\imp \big[P_{n+1}(y) + \b_nP_{n-1}(y)\big]^2\w(y)\,\dy\\ &\qquad
+ 6\imp \big[P_{n+2}(y) + (\b_{n+1}+\b_n)P_{n}(y)+\b_n\b_{n-1}P_{n-2}(y)\big]^2\w(y)\,\dy\\
&= (6x^4-2t)h_n + 6x^2(h_{n+1}+ \b_n^2h_{n-1})+6(h_{n+2} + (\b_{n+1}+\b_n)^2h_n+\b_n^2\b_{n-1}^2h_{n-2})\\
&= 6\big[x^4-\tfrac13t +x^2\big(\b_n+\b_{n+1}\big)+\b_{n+2}\b_{n+1}+\big(\b_{n+1}+\b_{n}\big)^2+\b_{n-1}\b_{n}\big]h_n,
\end{align*}
as required, since 
\[\imp {y P_n^2(y)}\,\w(y)\,\dy=\imp {y^3P_n^2(y)}\,\w(y)\,\dy=0\]
as these have odd integrands,
$\b_n=h_n/h_{n-1}$, the monic orthogonal polynomials $P_{n}(x)$ satisfy the three-term recurrence relation \eqref{eq:3rr},
and are orthogonal, i.e.
\beq \imp {P_m(y)}P_{n}(y)\,\w(y)\,\dy=0,\qquad{\rm if}\quad m\neq n.\label{Pnorth}\eeq 
Also for \eqref{int1b}
\begin{align*} \imp& \Kxy {P_n(y)P_{n-1}(y)}\,\w(y)\,\dy\nonumber\\ &=
(6x^4-2t) \imp {P_n(y)P_{n-1}(y)}\,\w(y)\,\dy 
+ 6x^3 \imp {yP_n(y)P_{n-1}(y)}\,\w(y)\,\dy \\ &\qquad\quad
+ 6x^2 \imp {y^2P_n(y)P_{n-1}(y)}\,\w(y)\,\dy + 6x \imp {y^3P_n(y)P_{n-1}(y)}\,\w(y)\,\dy \\&\qquad\quad
+ 6 \imp {y^4P_n(y)P_{n-1}(y)}\,\w(y)\,\dy \\
&= 6x^3\imp P_{n}(y)\big[P_{n}(y)+\b_{n-1}P_{n-2}(y)\big] \w(y)\,\dy \\
&\qquad\quad +6x \imp
\big[P_{n+2}(y)+\big(\b_{n+1}+\b_{n}\big)P_n(y)+\b_n\b_{n-1}P_{n-2}(y)\big] 
\big[P_{n}(y)+\b_{n-1}P_{n-2}(y)\big] \w(y)\,\dy\\
&=6x^3h_n+6x\big[\big(\b_{n+1}+\b_{n}\big)h_n+\b_n\b_{n-1}^2h_{n-2}\big]\\
&=6x\big(x^2+\b_{n+1}+\b_n+\b_{n-1}\big)h_n,
\end{align*}
as required, since
\[\begin{split}\imp {P_n(y)P_{n-1}(y)}\,\w(y)\,\dy=\imp y^2 {P_n(y)P_{n-1}(y)}\,\w(y)\,\dy=
\imp y^4 {P_n(y)P_{n-1}(y)}\,\w(y)\,\dy&=0\end{split}\]
as these have odd integrands, using the recurrence relation \eqref{eq:3rr} and orthogonality \eqref{Pnorth}.\end{proof}}

\begin{theorem}{For the generalised sextic Freud weight \eqref{genFreud6}
the monic orthogonal polynomials $S_{n}(x;t,\la)$ with respect to this weight satisfy the differential-difference equation
\beq \label{eq:Snddeq}
x\deriv{S_n}{x}(x;t,\la)=\mathcal{A}_n(x)S_{n-1}(x;t,\la)-\mathcal{B}_n(x)S_n(x;t,\la)
\eeq 
where
\begin{subequations}\label{AnBn2}\begin{align}
\mathcal{A}_n(x)&=6x\b_n\big[x^4-\tfrac13t +x^2\big(\b_n+\b_{n+1}\big)+\b_{n+2}\b_{n+1}+\big(\b_{n+1}+\b_{n}\big)^2+\b_{n-1}\b_{n}\big]\\
\mathcal{B}_n(x)&=6x^2\b_n\big(x^2+\b_{n+1}+\b_n+\b_{n-1}\big)+(\la+\tfrac12)[1-(-1)^n],
\end{align}\end{subequations}
with $\b_n$ the recurrence coefficient in the three-term recurrence relation \eqref{eq:3rr}.}\end{theorem}

\begin{proof}It was proved in \cite[Corollary 1]{refCJK} that monic orthogonal polynomials $S_n(x)=S_{n}(x;t,\la)$ with respect to the weight 
\[\w(x)=|x|^{2\la+1}\exp\{-v(x)\}\]
satisfy the differential-difference equation \eqref{eq:Snddeq}, where
\begin{align*}
\mathcal{A}_n(x)&=\frac{x}{h_{n-1}}\imp\Kxy {S_n^2(y)}\,\w(y)\,\dy\\
\mathcal{B}_n(x)&=\frac{x}{h_{n-1}}\imp\Kxy {S_n(y)S_{n-1}(y)} \,\w(y)\,\dy 
+(\la+\tfrac12)[1-(-1)^n]
\end{align*}
and \[\Kxy=\frac{\vx-\vy}{x-y}.\] For the generalised sextic Freud weight \eqref{genFreud6} we have
$v(x)=x^6-tx^2$ and hence
\[\Kxy= 6(x^4+x^3y+x^2y^2+xy^3+y^4)-2t.\]
Hence 
\begin{align*} \imp &\Kxy {S_n^2(y)}\,\w(y)\,\dy \\
&=
(6x^4-2t)\imp {S_n^2(y)}\,\w(y)\,\dy 
+ 6x^3\imp {yS_n^2(y)}\,\w(y)\,\dy + 6x^2\imp {y^2S_n^2(y)}\,\w(y)\,\dy\\
&\qquad + 6x\imp {y^3S_n^2(y)}\,\w(y)\,\dy + 6\imp {y^4S_n^2(y)}\,\w(y)\,\dy\\
&= (6x^4-2t)h_n +6x^2\imp \big[S_{n+1}(y) + \b_nS_{n-1}(y)\big]^2\w(y)\,\dy\\ &\qquad
+ 6\imp \big[S_{n+2}(y) + (\b_{n+1}+\b_n)S_n(y)+\b_n\b_{n-1}S_{n-2}(y)\big]^2\w(y)\,\dy\\
&= (6x^4-2t)h_n + 6x^2(h_{n+1}+ \b_n^2h_{n-1})+6(h_{n+2} + (\b_{n+1}+\b_n)^2h_n+\b_n^2\b_{n-1}^2h_{n-2})\\
&= 6\big[x^4-\tfrac13t +x^2\big(\b_n+\b_{n+1}\big)+\b_{n+2}\b_{n+1}+\big(\b_{n+1}+\b_{n}\big)^2+\b_{n-1}\b_{n}\big]h_n,
\end{align*}
since 
\[\imp {y S_n^2(y)}\,\w(y)\,\dy=\imp {y^3S_n^2(y)}\,\w(y)\,\dy=0\]
as these have odd integrands,
$\b_n=h_n/h_{n-1}$ where 
\beq \nonumber \label{def:hn} h_n = \imp S_n^2(y)\,\w(y)\,\dy\eeq 
the monic orthogonal polynomials $S_{n}(x)$ satisfy the three-term recurrence relation \eqref{eq:3rr},
and are orthogonal, i.e.
\beq \imp {S_m(y)}S_n(y)\,\w(y)\,\dy=0,\qquad{\rm if}\quad m\neq n.\label{Pnorth}\eeq 
Also,
\begin{align*} \imp& \Kxy {S_n(y)S_{n-1}(y)}\,\w(y)\,\dy\nonumber\\ &=
(6x^4-2t) \imp {S_n(y)S_{n-1}(y)}\,\w(y)\,\dy 
+ 6x^3 \imp {yS_n(y)S_{n-1}(y)}\,\w(y)\,\dy \\ &\qquad\quad
+ 6x^2 \imp {y^2S_n(y)S_{n-1}(y)}\,\w(y)\,\dy + 6x \imp {y^3S_n(y)S_{n-1}(y)}\,\w(y)\,\dy \\&\qquad\quad
+ 6 \imp {y^4S_n(y)S_{n-1}(y)}\,\w(y)\,\dy \\
&= 6x^3\imp S_n(y)\big[S_n(y)+\b_{n-1}S_{n-2}(y)\big]\w(y)\,\dy \\
&\qquad\quad +6x \imp
\big[S_{n+2}(y)+\big(\b_{n+1}+\b_{n}\big)S_n(y)+\b_n\b_{n-1}S_{n-2}(y)\big] 
\big[S_n(y)+\b_{n-1}S_{n-2}(y)\big]\w(y)\,\dy\\
&=6x^3h_n+6x\big[\big(\b_{n+1}+\b_{n}\big)h_n+\b_n\b_{n-1}^2h_{n-2}\big]\\
&=6x\big(x^2+\b_{n+1}+\b_n+\b_{n-1}\big)h_n,
\end{align*}
since
\[\begin{split}\imp {S_n(y)S_{n-1}(y)}\,\w(y)\,\dy=\imp y^2 {S_n(y)S_{n-1}(y)}\,\w(y)\,\dy=
\imp y^4 {S_n(y)S_{n-1}(y)}\,\w(y)\,\dy&=0\end{split}\]
as these have odd integrands, using the recurrence relation \eqref{eq:3rr} and orthogonality \eqref{Pnorth}.\end{proof}

\begin{remark} In \cite{refWZC20a}, the ladder operator technique was used to obtain the coefficients ${A}_n$ and ${B}_n$ in \eqref{ddee} for the weight \eqref{genFreud6}. Note however that, in their notation, the expression for ${B}_n$ (cf.~\cite[eqn. (39)]{refWZC20a}) should be \[{B}_n(z)=6z^3\b_n +6z\b_n\big(\b_{n+1}+\b_n+\b_{n-1}\big)+\frac{\a[1-(-1)^n]}{2z}\]

\end{remark}

Now we derive a differential equation satisfied by generalised sextic Freud polynomials.
\begin{theorem} \label{thm:gende}For the generalised sextic Freud weight \eqref{genFreud6}
the monic orthogonal polynomials $S_{n}(x;t,\la)$ with respect to this weight satisfy 
\[
x\deriv[2]{S_n}{x}(x;t,\la)+Q_n(x)\deriv{S_n}{x}(x;t,\la)+T_n(x)S_n(x;t,\la)=0
\]
where
\begin{align*}
Q_n(x)=&~2 tx^2 -6 x^6+2 \la +1
-\frac{2 x^2 \left(2x^2+\b_n+\b_{n+1}\right)}{C_n(x)}\\
T_n(x)=&~36 x \b_{n} C_{n-1}(x)C_n(x)+12 x^3 \b_{n}+\frac{(2 \la +1)}{x} \left\{6 x^2 \b_{n} D_n(x)+( \la +\tfrac12) [1-(-1)^n-1]\right\}+12 x \b_{n} D_n(x)\\
&-\left\{6 x^2 \b_{n} D_n(x)+(\la +\tfrac12) [1-(-1)^n-1]\right\}\left\{6 x \b_{n} D_n(x)+\frac{(2 \la +1)}{2x} [1-(-1)^n-1]-2 t x+6 x^5\right\}\\&-\frac{\left\{C_n(x)+4x^4+2 x^2(\b_n+\b_{n+1})\right\} \left\{6 x^2 \b_{n} D_n(x)+( \la +\tfrac12) [1-(-1)^n-1]\right\}}{x C_n(x)},
\end{align*}
with 
\begin{align*}C_n(x)&=x^4-\tfrac13t +x^2\big(\b_n+\b_{n+1}\big)+\b_{n+2}\b_{n+1}+\big(\b_{n+1}+\b_{n}\big)^2+\b_{n-1}\b_{n}\\
D_n(x)&=x^2+\b_{n-1}+\b_{n}+\b_{n+1}.\end{align*}
\comment{\begin{align*}
\mathcal{A}_n(x)&=\frac{x}{h_{n-1}}\imp S_n^2(y;t,\la)\,\Kxy\,w(y;t,\la)\,\dy\\
\mathcal{B}_n(x)&= \frac{x}{h_{n-1}}\imp S_n(y;t,\la)S_{n-1}(y;t,\la)\,\Kxy\,\w(y;t,\la)\,\dy 
(\la+\tfrac12)[1-(-1)^n].
\end{align*}}
\end{theorem}
\begin{proof} In \cite[Theorem 3]{refCJK} it was proved that the coefficients in the differential equation 
\[x\deriv[2]{S_n}{x}(x)+Q_n(x)\deriv{S_n}{x}(x)+T_n(x)S_n(x)=0\]
satisfied by polynomials orthogonal with respect to the weight \[\ww{x}=|x|^{\rho}\exp\{-\v(x;t)\}\] are given by 
\begin{subequations}\label{coeff12}\begin{align}
Q_n(x)&=\rho+1-x\deriv{\v}{x}-\frac{x}{\mathcal{A}_n(x)}\,\deriv{\mathcal{A}_n}{x}\\[5pt]
T_n(x)&=\frac{\mathcal{A}_n(x)\mathcal{A}_{n-1}(x)}{x\b_{n-1}}+\deriv{\mathcal{B}_n}{x} 
-\mathcal{B}_n(x)\left[\deriv{\v}{x} +\frac{\mathcal{B}_n(x)-\rho}{x}\right]-\frac{\mathcal{B}_n(x)}{\mathcal{A}_n(x)}\deriv{\mathcal{A}_n}{x},
\end{align}
with
\begin{align*}
\mathcal{A}_n(x)&=\frac{x}{h_{n-1}}\imp S_n^2(y;t,\la)\,\Kxy\,\w(y;t,\la)\,\dy\\
\mathcal{B}_n(x)&= \frac{x}{h_{n-1}}\imp S_n(y;t,\la)S_{n-1}(y;t,\la)\,\Kxy\,\w(y;t,\la)\,\dy 
+\tfrac12{\rho}[1-(-1)^n].
\end{align*}\end{subequations}
For the generalised sextic Freud weight \eqref{genFreud6} we use \eqref{coeff12} with $k=0$, $\rho=2\la+1$ and $\v(x)=x^6-tx^2$ to obtain 
\begin{subequations}
\begin{align}\label{coef1}
Q_n(x)&=2\la+2-6x^6+2tx^2 -\frac{x}{\mathcal{A}_n(x)}\,\deriv{\mathcal{A}_n}{x}\\[5pt]
\label{coef2}T_n(x)&=\frac{\mathcal{A}_n(x)\mathcal{A}_{n-1}(x)}{x\b_{n-1}}+\deriv{\mathcal{B}_n}{x} 
-\mathcal{B}_n(x)\left[6x^5-2tx +\frac{\mathcal{B}_n(x)-(2\la+1) }{x}\right]-\frac{\mathcal{B}_n(x)}{\mathcal{A}_n(x)}\deriv{\mathcal{A}_n}{x}.
\end{align}\end{subequations}
Substituting the expressions for $\mathcal{A}_n(x)$ and $\mathcal{B}_n(x)$ given by \eqref{AnBn2} and their derivatives into \eqref{coef1} and \eqref{coef2}, we obtain the stated result on simplification.
\end{proof}
\begin{lemma}\label{mrec} Let $\{S_n(x;t,\la)\}_{n=0}^{\infty}$ be the sequence of monic generalised sextic Freud polynomials orthogonal with respect to the weight
\eqref{genFreud6}, then, for $n$ fixed, 
\begin{align}\label{l+2}
x^2S_n(x;t,\la+1)=xS_{n+1}(x;t,\la)-(\b_{n+1}+a_n)S_n(x;t,\la)
\end{align} where 
\beq \nonumber a_n=\begin{cases} \ds\frac{S_{n+2}(0;t,\la)}{S_n(0;t,\la)},\quad &\text{if}\quad n\quad\text{even}\\[5pt]
\ds\frac{S_{n+2}'(0;t,\la)}{S_n'(0;t,\la)},\quad &\text{if}\quad n\quad\text{odd}.\end{cases}\label{ed}\eeq
\end{lemma}

\begin{proof}The weight function associated with the polynomials $S_n(x;t,\la+1)$ is 
\begin{align*}\w(x;t,\la+1)&=|x|^{2\la+3}\exp(-x^6+tx^2) 
=x^2\w(x;t,\la).\end{align*} 
The factor $x^2$ by which the weight $\w(x;t,\la)$ is modified has a double zero at the origin and therefore Christoffel's formula (cf.~\cite[Theorem 2.5]{refSzego}), applied to the monic polynomials $S_n(x;t,\la+1)$, is
\[
x^2S_{n}(x;t,\la+1)=\frac{1}{S_n(0;t,\la)S_{n+1}'(0;t,\la)-S_n'(0;t,\la)S_{n+1}(0;t,\la)}\left|\begin{matrix} S_{n}(x;t,\la) & S_{n+1}(x;t,\la) & S_{n+2}(x;t,\la)\\
S_{n}(0;t,\la) & S_{n+1}(0;t,\la) & S_{n+2}(0;t,\la)\\
S_{n}'(0;t,\la) & S_{n+1}'(0;t,\la) & S_{n+2}'(0;t,\la)\\\end{matrix}\right|
\]
Since the weight $\w(x;t,\la)$ is even, we have that $S_{2n+1}(0;t,\la)=S_{2n}'(0;t,\la)=0$ while $S_{2n}(0;t,\la)\neq0$ and $S_{2n+1}'(0;t,\la)\neq0$, hence 
\[
x^2S_{n}(x;t,\la+1)=\frac{-1}{ S_n'(0;t,\la)S_{n+1}(0;t,\la)}\left|\begin{matrix} S_{n}(x;t,\la) & S_{n+1}(x;t,\la) & S_{n+2}(x;t,\la)\\
0 & S_{n+1}(0;t,\la) &0\\
S_{n}'(0;t,\la) & 0 & S_{n+2}'(0;t,\la)\\\end{matrix}\right|
\] for $n$ odd, while, for $n$ even, 
\[
x^2S_{n}(x;t,\la+1)=\frac{1}{S_n(0;t,\la)S_{n+1}'(0;t,\la) }\left|\begin{matrix} S_{n}(x;t,\la) & S_{n+1}(x;t,\la) & S_{n+2}(x;t,\la)\\
S_{n}(0;t,\la) & 0& S_{n+2}(0;t,\la)\\
0 & S_{n+1}'(0;t,\la) &0\\\end{matrix}\right|
\]
This yields
\beq\label{even}
x^2S_n(x;t,\la+1)=S_{n+2}(x;t,\la)-a_nS_n(x;t,\la)
\eeq and the result follows by using the three-term recurrence relation \eqref{eq:3rr} to eliminate $S_{n+2}(x;t,\la)$ in \eqref{even}.
\comment{\item[(ii)] Replacing $n$ by $n-2$ in \eqref{l+2} and then using the three-term recurrence relation \[S_{n-2}(x;t,\la)=\frac{1}{\b_{n-1}}\left(xS_{n-1}(x;t,\la)-S_n(x;t,\la)\right)\] to eliminate $S_{n-2}(x;t,\la)$ yields the result.} 
\end{proof}

\subsection{Zeros of generalised sextic Freud polynomials}
\noindent When the weight is even, the zeros of the corresponding orthogonal polynomials are symmetric about the origin. This implies that the positive and the negative zeros have opposing monotonicity and therefore we only need to consider the monotonicity of the positive zeros.
\begin{lemma}\label{symmon}
Let $\w_0(x)$ be a symmetric positive weight on $(a,b)$ for which all the moments exist and let\beq\label{genws}\w(x;t,\rho)=|C(x)|^\rho \exp\{tD(x)\} \w_0(x),\qquad t\in \R, \qquad \rho>-1\eeq where $D(x)$ is an even function. Consider the sequence of semiclassical orthogonal polynomials $\{S_n(x;t,\rho)\}_{n=0}^{\infty}$ associated with the weight \eqref{genws} and denote the $\lfloor n/2\rfloor$ real, positive zeros of $S_n(x;t,\gamma)$ in increasing order by $x_{n,k}(t,\gamma)$, $k=1,2,\dots,\lfloor n/2\rfloor$,
where $\lfloor m \rfloor$ is the largest integer less than or equal to  $m$. 
Then, for a fixed value of $\nu$, $\nu\in \{1,2,\dots, \lfloor n/2 \rfloor\}$, the $\nu$-th zero $x_{n,\nu}(\la,t)$
\begin{itemize}
\item[(i)] increases when $\rho$ increases, if $\displaystyle{\frac{1}{C(x)}\deriv{}{x}C(x)>0}$ for $x\in (0,b)$;\\[-0.4cm]
\item[(ii)] increases when $t$ increases, if $\displaystyle{\deriv{}{x}D(x)>0}$ for $x\in (0,b)$;
\end{itemize}
\end{lemma}
\begin{proof} The proof follows along the same lines as that of Lemma \ref{mono} using the generalised version of Markov's monotonicity theorem (cf.~\cite[Theorem 2.1]{refJWZ}) for $x>0$.
\end{proof}
\begin{corollary}Let $\{S_n(x;t,\la)\}_{n=0}^{\infty}$ be the sequence of monic generalised sextic Freud polynomials orthogonal with respect to the weight \eqref{genFreud6} and let $0<x_{\lfloor n/2 \rfloor,n}<\dots<x_{2,n} <x_{1,n} $ denote the positive zeros of $S_n(x;t,\la)$.
Then, for $\la>-1$ and $t\in \R$ and for a fixed value of $\nu$, $\nu\in \{1,2,\dots, \lfloor n/2 \rfloor\}$, the $\nu$-th zero $x_{n,\nu} $ increases when  (i) $\la$ increases; and (ii)
$t$ increases.
\end{corollary} 
\begin{proof} This follows from Lemma \ref{symmon}, taking $C(x)=x$, $D(x)=x^2$, $\rho=2\la+1$ and $\w_0(x)=\exp(-x^6)$.
\end{proof}

%

Mixed recurrence relations involving polynomials from different orthogonal sequences, such as the relation derived in Lemma \ref{mrec}, provide information on the relative positioning of zeros of the polynomials in the relation. In the next theorem we prove that the zeros of $S_n(x;t,\la)$, the monic generalised sextic Freud polynomials orthogonal with respect to the weight
\eqref{genFreud6}, and the zeros of $S_{n-1}(x;t,\la+k)$ interlace for $\la>-1$, $t\in \R$ and $k\in (0,1]$ fixed.

\begin{theorem} \label{int} Let $\la>-1$, $t\in \R$ and $k\in (0,1)$. Let $\{S_n(x;t,\la)\}$ be the monic generalised sextic Freud polynomials orthogonal with respect to the weight
\eqref{genFreud6}. Denote 
%
the positive zeros of $S_{n}(x;t,\la+k)$ by \[0<x_{\lfloor \frac n2\rfloor,n}^{(t,\la+k)}<x_{\lfloor \frac n2\rfloor -1,n}^{(t,\la+k)}<\dots<x_{2,n}^{(t,\la+k)}<x_{1,n }^{(t,\la+k)}.\]

\noindent If $n$ is even, then
\begin{align}\label{inteven}
0<&x_{\lfloor \frac{n}{2}\rfloor,n}^{(t;\la)}<x_{\lfloor \frac{n-1}{2} \rfloor,n-1}^{(t;\la)}<x_{\lfloor \frac{n-1}{2} \rfloor,n-1}^{(t,\la+k)}<x_{\lfloor \frac{n-1}{2} \rfloor,n-1}^{(t,\la+1)}<x_{\lfloor \frac{n}{2} \rfloor-1,n}^{(t;\la)}<\dots\nonumber\\ &\dots <x_{2,n}^{(t;\la)}<x_{1,n-1}^{(t;\la)}<x_{1,n-1}^{(t,\la+k)}<x_{1,n-1}^{(t,\la+1)}<x_{1,n}^{(t;\la)}
\end{align}
and if $n$ is odd, then
\begin{align}\label{intodd}
0<&x_{\lfloor \frac{n-1}{2} \rfloor,n-1}^{(t;\la)}<x_{\lfloor \frac{n-1}{2} \rfloor,n-1}^{(t,\la+k)}<x_{\lfloor \frac{n-1}{2} \rfloor,n-1}^{(t,\la+1)}<x_{\lfloor \frac{n}{2} \rfloor,n}^{(t;\la)}<x_{\lfloor \frac{n-1}{2} \rfloor-1,n-1}^{(t;\la)}<\dots\nonumber\\ &\dots <x_{2,n}^{(t;\la)}<x_{1,n-1}^{(t;\la)}<x_{1,n-1}^{(t,\la+k)}<x_{1,n-1}^{(t,\la+1)}<x_{1,n}^{(t;\la)}.
\end{align}\end{theorem}

\begin{proof} In Theorem \ref{mono} we proved that the positive zeros of $S_{n-1}(x;t,\la)$ monotonically increase as $\la$ increases. This implies that, for each fixed $\ell\in\{1,2,\dots,\lfloor \frac{n-1}{2}\rfloor\}$, 
\beq\label{mo} x_{ \ell ,n-1}^{(t;\la)}<x_{\ell,n-1}^{(t,\la+k)}<x_{ \ell,n -1}^{(t,\la+1)}.\eeq 
On the other hand, the zeros of $S_{n}(x;t,\la)$ and $S_{n-1}(x;t,\la)$, two consecutive polynomials in the sequence of orthogonal polynomials, are interlacing, that is, when $n$ is even, 
\beq\label{3.13b}
0<x_{\lfloor \frac{n}{2}\rfloor,n}^{(t;\la)}<x_{\lfloor\frac{n-1}{2}\rfloor,n-1}^{(t;\la)}<x_{\lfloor\frac{n}{2}\rfloor-1,n}^{(t;\la)}<\dots<x_{2,n}^{(t;\la)}<x_{1,n-1}^{(t;\la)}<x_{1,n}^{(t;\la)}.
\eeq
Next, we prove that the zeros of $S_{n}(x;t,\la)$ interlace with those of $S_{n-1}(x;t,\la+1)$. Replacing $n$ by $n-1$ in \eqref{l+2} yields
\beq\label{l+22}
S_{n-1}(x;t,\la+1)=\frac{xS_{n}(x;t,\la)-(\b_{n}+a_{n-1})S_{n-1}(x;t,\la)}{x^2}.
\eeq
Evaluating \eqref{l+22} at consecutive zeros $x_{\ell}=x_{\ell,n}^{(t;\la)}$ and $x_{\ell+1}=x_{\ell+1,n}^{(t;\la)}$, $\ell=1, 2, \ldots, \lfloor \frac{n}{2}\rfloor-1$, of $S_{n}(x;t,\la)(x)$, we obtain
\[ 
S_{n-1}(x_{\ell};t,\la+1)S_{n-1}(x_{\ell+1};t,\la+1)=\frac{1}{x_{\ell}^2 x_{\ell+1}^2}(\b_{n}+a_{n-1})^2S_{n-1}(x_{\ell};t,\la)S_{n-1}(x_{\ell+1};t,\la)<0
\]
since the zeros of $S_{n}(x;t,\la)$ and $S_{n-1}(x;t,\la)$ seperate each other. So there is at least one positive zero of $S_{n}(x;t,\la+1)$ in the interval $(x_{\ell}, x_{\ell+1})$ for each $\ell=1, 2, \ldots, \lfloor \frac{n}{2}\rfloor-1$
and this implies that
\begin{align}\label{3.14b}
0<x_{\lfloor \frac{n}{2}\rfloor,n}^{(t;\la)}<x_{\lfloor \frac{n-1}{2} \rfloor,n-1}^{(t,\la+1)}<x_{\lfloor \frac{n}{2}\rfloor -12,n}^{(t;\la)}<x_{\lfloor \frac{n-1}{2} \rfloor-1,n-1}^{(t,\la+1)}<
\dots<x_{2,n-1}^{(t,\la+1)}<x_{2,n}^{(t;\la)}<x_{1,n-1}^{(t,\la+1)}<x_{1 ,n}^{(t;\la)}
\end{align}
\eqref{3.14b}, \eqref{mo} and \eqref{3.13b} yield \eqref{inteven}. The proof of \eqref{intodd} follows along the same lines.\end{proof}

Considering that when the weight function is even, the zeros of $S_n(x;t,\la)$ are symmetric about the origin with a zero at the origin when $n$ is odd, we have the following corollary.

\begin{corollary}\label{cor2b} With the same symbols as Theorem \ref{int}, we have for $n$ odd that

\[ x_{n,n}^{(t;\la)}<x_{n-1,n-1}^{(t;\la)}<x_{n-1,n-1}^{(t,\la+k)}<x_{n-1,n-1}^{(t,\la+1)}<x_{n-1,n}^{(t;\la)}<\dots<x_{2,n}^{(t;\la)}<x_{1,n-1}^{(t;\la)}<x_{1,n-1}^{(t,\la+k)}<x_{1,n-1}^{(t,\la+1)}<x_{1,n}^{(t;\la)} \]
while for $n$ even
\[x_{n,n}^{(t;\la)}<x_{n-1,n-1}^{(t;\la)}<x_{n-1,n-1}^{(t,\la+k)}<x_{n-1,n-1}^{(t,\la+1)}<x_{n-1,n}^{(t;\la)}<\dots<x_{\lfloor\frac{n-1}{2}\rfloor+2,n-1}^{(t,\la+1)}<x_{\lfloor\frac{n}{2}\rfloor+1,n}^{(t;\la)}<0\]
and \[0<x_{\lfloor\frac{n}{2}\rfloor,n}^{(t;\la)}<x_{\lfloor\frac{n-1}{2}\rfloor,n-1}^{(t;\la)}<x_{\lfloor\frac{n-1}{2}\rfloor,n-1}^{(t,\la+k)}<
x_{\lfloor\frac{n-1}{2}\rfloor,n-1}^{(t,\la+1)}<x_{\lfloor\frac{n}{2}\rfloor-1,n}^{(t;\la)}<\dots<x_{1,n-1}^{(t,\la+1)}<x_{1,n}^{(t;\la)} \]
with \[x_{\lfloor\frac{n-1}{2}\rfloor+1,n-1}^{(t;\la)}=x_{\lfloor\frac{n-1}{2}\rfloor+1,n-1}^{(t,\la+k)}=x_{\lfloor\frac{n-1}{2}\rfloor+1,n-1}^{(t,\la+1)}=0.\]
\end{corollary}
\comment{\subsection{Bounds for the zeros}}

The three-term recurrence relation yields information on bounds of the extreme zeros of polynomials.

\begin{theorem}Let $\{S_n(x;t,\la)\}_{n=0}^{\infty}$ be the sequence of monic generalised sextic Freud polynomials orthogonal with respect to the weight \eqref{genFreud6}. For each $n=2,3,\dots,$ the largest zero, $x_{1,n}$, of $S_n(x;t,\la)$, satisfies 
\[0<x_{1,n} <\max_{1\leq k\leq n-1}\sqrt{c_n\b_k(t;\la)}\]
where $c_n=4\cos^2\left(\frac{\pi}{n+1}\right)+\ep$, $\ep>0$.

\end{theorem}
\begin{proof} 
The upper bound for the largest zero $x_{1,n}$ follows by applying \cite[Theorem 2 and 3]{refIsmailLi}, based on the Wall-Wetzel Theorem (see also \cite{refIsmail}), to the three-term recurrence relation \eqref{eq:3rr}. 
\end{proof}

\comment{\subsection{Convexity of the zeros when \boldmath{$\la=-\tfrac12$}}}

The Sturm Convexity Theorem (cf.~\cite{refSturm}) on the monotonicity of the distances between consecutive zeros, applies to the zeros of solutions of second-order differential equations in the normal form 
\[\deriv[2]{y}{x}+F(x)y=0. \]
Next we consider the implications of the convexity theorem of Sturm for the zeros of generalised sextic Freud polynomials when $\la=-\tfrac12$. We begin by considering the differential equation in normal form satisfied by generalised sextic Freud polynomials for $\la=-\tfrac12$ proved by Wang \etal\ in \cite{refWZC20a}.
\begin{theorem}Let 
\beq
\label{l0w}
\w(x)=\exp(-x^6+tx^2),\qquad x\in\R\eeq with $t\in\R$, and denote the monic orthogonal polynomials with respect to $\w(x)$ by $\mathcal{S}_n(x)$. Then, for $t<0$, the polynomials\beq\label{trans}S_n(x;t,\la)=\mathcal{S}_n(x)\sqrt{\frac{\w(x)}{\widetilde{A}_n(x)}}\eeq satisfy
\beq\label{eq:gende}
\deriv[2]{S_n}{x}+F(x)S_n(x;t,\la)=0
\eeq
where
\begin{align}\label{Fc}
F(x)&=\b_n \widetilde{A}_{n-1}(x)\widetilde{A}_n(x)-\frac{\w''(x)}{2\w(x)}-\widetilde{B}_n(x)\left\{\widetilde{B}_n(x)-6x^5+2tx\right\}+\frac{6x^5-2tx}{4}-3\left\{\frac{\widetilde{A}_n'(x)}{2\widetilde{A}_n(x)}\right\}^2\\&\qquad\nonumber+\widetilde{B}_n'(x)-\frac{(2\widetilde{B}_n(x)+6x^5-2tx)\widetilde{A}_n'(x)-\widetilde{A}_n''(x)}{2\widetilde{A}_n(x)}\\
\widetilde{A}_n(x)&=\frac{\mathcal{A}_n(x)}{x\b_n}=6\big[x^4-\tfrac13t +x^2\big(\b_n+\b_{n+1}\big)+\b_{n+2}\b_{n+1}+\big(\b_{n+1}+\b_{n}\big)^2+\b_{n-1}\b_{n}\big]\nonumber\\
\widetilde{B}_n(x)&=\frac{\mathcal{B}_n(x)}{x}=6x\b_n\big(x^2+\b_{n+1}+\b_n+\b_{n-1}\big)+(\la+\tfrac12)[1-(-1)^n].\nonumber
\end{align}

\end{theorem}
\begin{proof} 
See \cite[Theorem 4]{refWZC20a} and note that $\widetilde{A}_n>0$ when $t<0$. 
\end{proof}

\begin{theorem}
Let $\{\mathcal{S}_n(x)\}_{n=0}^{\infty}$ be the monic generalised sextic Freud polynomials orthogonal with respect to the weight
\eqref{l0w}
and let $x_{k}$, $k\in\{1,2,\dots,n\}$, denote the $n$ zeros of $\mathcal{S}_n$ in ascending order. Then, for $t>0$, 
\begin{itemize}
\item[(i)] if $F(x)$ given in \eqref{Fc} is strictly increasing on $(a,b)$, then, for the zeros $x_k \in (a,b)$, we have $x_{k+2}-x_{k+1}<x_{k+1}-x_k$, i.e. the zeros in $(a,b)$ are concave; 
\item[(ii)] if $F(x)$ given in \eqref{Fc} is strictly decreasing on $(a,b)$, then, for the zeros $x_k\in(a,b)$, we have $x_{k+2}-x_{k+1}>x_{k+1}-x_k$, i.e. the zeros in $(a,b)$ are convex.\end{itemize}
\end{theorem}
\begin{proof}
Since the transformation \eqref{trans} does not change the independent variable and $\w(x)>0$, the zeros of $\mathcal{S}_n(x)$ are the same as those of $S_n(x;t,\la)$. The result now follows by applying the Sturm convexity Theorem (cf.~\cite{refJT,refSturm}) to solutions of \eqref{eq:gende}.
\end{proof}
\section{Discussion}
In this paper, we have studied generalised Airy polynomials that are orthogonal polynomials that satisfy a three-term recurrence relation whose coefficients depend on two parameters. We have derived a differential difference equation, a differential equation and a mixed recurrence relation satisfied by the polynomials and used these to study properties of the zeros and recurrence coefficients of the polynomials. We also investigated various asymptotic properties of the recurrence coefficients. Furthermore we have shown that similar results hold for the generalised sextic Freud polynomials and corrected some results in the literature.   
\section*{Acknowledgements} 
We gratefully acknowledge the support of a Royal Society Newton Advanced Fellowship NAF$\backslash$R2$\backslash$180669.
PAC would like to thank the Isaac Newton Institute for Mathematical Sciences for support and hospitality during the programme ``Complex analysis: techniques, applications and computations" when some of the work on this paper was undertaken. This work was supported by EPSRC grant number EP/R014604/1. We also thank the referees for helpful comments and corrections.

\def\cpam{Commun. Pure Appl. Math.}
\def\CPAM{Commun. Pure Appl. Math.}
\def\funk{Funkcial. Ekvac.}
\def\FUNK{Funkcial. Ekvac.}
\def\CUP{Cambridge University Press}

\def\refpp#1#2#3#4{\vspace{-0.2cm}
\bibitem{#1} \textrm{\frenchspacing#2}, \textrm{#3}, #4.}

\def\refjnl#1#2#3#4#5#6#7{\vspace{-0.2cm}
\bibitem{#1}{\frenchspacing\rm#2}, #3, 
{\frenchspacing\it#4}, {\bf#5} (#6) #7.}

\def\refjltoap#1#2#3#4#5#6#7{\vspace{-0.2cm}
\bibitem{#1}\textrm{\frenchspacing#2}, \textrm{#3},
\textit{\frenchspacing#4}\ (#6)\ #7.}

\def\refjl#1#2#3#4#5#6#7{\vspace{-0.2cm}
\bibitem{#1}\textrm{\frenchspacing#2}, \textrm{#3},
\textit{\frenchspacing#4}, \textbf{#5}\ (#7)\ #6.}
\def\refbk#1#2#3#4#5{\vspace{-0.2cm}
\bibitem{#1} \textrm{\frenchspacing#2}, \textit{#3}, #4, #5.}

\def\refcf#1#2#3#4#5#6#7{\vspace{-0.2cm}
\bibitem{#1} \textrm{\frenchspacing#2}, \textrm{#3},
in: \textit{#4}, {\frenchspacing#5}, #6, pp.\ #7.}
\def\ams{American Mathematical Society}
\def\DE{Diff. Eqns.}
\def\JPA{J. Phys. A}
\def\NMJ{Nagoya Math. J.}
\def\JCAM{J. Comput. Appl. Math.}

\end{document}